\documentclass{AIMS}

  \usepackage{paralist}
\usepackage{amsmath,amssymb,graphicx,amsthm,mathrsfs}
 \usepackage[colorlinks=true]{hyperref}
\hypersetup{urlcolor=blue, citecolor=red}

\def\currentvolume{X}
 \def\currentissue{X}
  \def\currentyear{200X}
   \def\currentmonth{XX}
    \def\ppages{X--XX}
     \def\DOI{10.3934/xx.xx.xx.xx}

\newtheorem{theorem}{Theorem}[section]
\newtheorem{lemma}[theorem]{Lemma}
\newtheorem{proposition}[theorem]{Proposition}
\newtheorem{corollary}[theorem]{Corollary}
\newtheorem*{theore}{Theorem \ref*{generator}}
\theoremstyle{definition}
\newtheorem{definition}[theorem]{Definition}
\newtheorem{remark}[theorem]{Remark}
\newtheorem{assumption}[theorem]{Assumption}

\numberwithin{equation}{section}

\newcommand{\mh}{\ensuremath{\mathcal{H}}}
\newcommand{\A}{\ensuremath{\mathcal{A}}}
\newcommand{\HH}{\mathcal{H}}
\newcommand{\XX}{\mathcal{X}}
\newcommand{\ma}{\ensuremath{\mathcal{A}}}
\newcommand{\mb}{\ensuremath{\mathcal{B}}}
\newcommand{\mth}{\ensuremath{\tilde{\mathcal{H}}}}
\newcommand{\nn}{\mathcal N}
\newcommand{\C}{\mathbb{C}}
\newcommand{\R}{\mathbb{R}}

\newcommand{\N}{\mathbb{N}}
\newcommand{\Z}{\mathbb{Z}}
\newcommand{\dd}{\,\mathrm{d}}
\newcommand{\e}{\mathrm{e}}
\newcommand{\Real}{\operatorname{Re}}
\newcommand{\ii}{\mathrm{i}}
\newcommand{\inj}{\hookrightarrow}
\newcommand{\spn}{\operatorname{span}}
\newcommand{\la}{\langle}
\newcommand{\ra}{\rangle}
\newcommand{\laa}{\la\!\la}
\newcommand{\raa}{\ra\!\ra}

\newcommand{\ran}{\operatorname{ran}}
\newcommand{\Ker}{\operatorname{ker}}
\newcommand{\rom}[1]{\textrm{\tiny{\uppercase\expandafter{\romannumeral #1}}}}

\title[Euler-Bernoulli beam with a nonlin.~boundary control]{An Euler-Bernoulli beam with nonlinear damping and a nonlinear spring at the tip}

\author[M.~Mileti\'c, D.~St\"urzer and A.~Arnold]{}


\email{mmiletic@asc.tuwien.ac.at}
\email{dominik.stuerzer@tuwien.ac.at}
\email{anton.arnold@tuwien.ac.at}

\thanks{This research was supported by the FWF-doctoral school ``Dissipation and dispersion in nonlinear partial differential equations'' and the FWF-project I395-N16. Two authors (AA, MM) acknowledge a sponsorship by \emph{Clear Sky Ventures}. The authors want to thank A.~Kugi for interesting discussions that led to the formulation of the model considered here.}

\begin{document}
\maketitle

% Enter the first author's name and address:
\centerline{\scshape Maja Mileti\'c, Dominik St\"urzer and Anton Arnold}
\medskip
{\footnotesize
% please put the address of the first author
 \centerline{Institute for Analysis and
 Scientific Computing}
 \centerline{Vienna University of Technology}
   \centerline{Wiedner
 Hauptstra\ss{}e 8}
   \centerline{1040 Vienna, Austria}
} % Do not forget to end the {\footnotesize by the sign }

%\medskip
%
%\centerline{\scshape First-name2 last-name2 and First-name3
%last-name3}
%\medskip
%{\footnotesize
% % please put the address of the second  and third author
% \centerline{ First line of the address of the second author}
%   \centerline{Other lines}
%   \centerline{Springfield, MO 65810, USA}
%}

\bigskip

% The name of the associate editor will be entered by an editorial staff
% "Communicated by the associate editor name" is not needed for special issue.
 \centerline{(Communicated by the associate editor name)}

%The abstract of your paper
\begin{abstract}
We study the asymptotic behavior for a system consisting of a clamped flexible beam that carries a tip payload, which is attached to a nonlinear damper and a nonlinear spring at its end. Characterizing the $\omega$-limit sets of the trajectories, we give a sufficient condition under which the system is asymptotically stable.  In the case when this condition is not satisfied, we show that the beam deflection approaches a non-decaying time-periodic solution.
\end{abstract}

%
%\title[]{An Euler-Bernoulli beam with nonlinear damping and a nonlinear spring at the tip}
%
%\author[M.~Mileti\'c]{Maja Mileti\'c} \address{Institute for Analysis and
% Scientific Computing, Technical University Vienna, Wiedner
% Hauptstra\ss{}e 8, A-1040 Vienna}
%\email{mmiletic@asc.tuwien.ac.at}
%
%\author[D.~St\"urzer]{Dominik St\"urzer} \address{Institute for Analysis and
% Scientific Computing, Technical University Vienna, Wiedner
% Hauptstra\ss{}e 8, A-1040 Vienna}
%\email{dominik.stuerzer@tuwien.ac.at}
%
%\author[A.~Arnold]{Anton Arnold} \address{Institute for Analysis and
% Scientific Computing, Technical University Vienna, Wiedner
% Hauptstra\ss{}e 8, A-1040 Vienna}
%\email{anton.arnold@tuwien.ac.at}
%
%
%
%\begin{abstract}
%We study the asymptotic behaviour for a system consisting of a clamped flexible beam that carries a tip payload, which is attached to a nonlinear damper and a nonlinear spring at its end. Characterizing the $\omega$-limit sets of the trajectories, we give a necessary condition under which the system is asymptotically stable.  In the case when this condition is not satisfied, we show that the beam deflection approaches a non-decaying time-periodic solution.
%\end{abstract}
%
%
%
%
%\subjclass[2010]{35B40,70K20,74K10,47H20, 35Q70 }
%\keywords{flexible beam, nonlinear spring, nonlinear damping, nonlinear semigroups, trajectory precompactness, asymptotic stability }
%\maketitle
%
%

\section{Introduction}
This article considers an Euler-Bernoulli beam where one end is clamped, and the free end holds a rigid tip mass. Models of this form play a fundamental role in many mechanical systems and thus occur in many applications such as flexible robot arms, helicopter rotor blades, spacecraft antennae, airplane wings, high-rise buildings, etc. An important issue is the suppression of vibrations, since undesired oscillations can reduce the performance of the system, or worse, result in damage to the structure. For this reason, the Euler-Bernoulli beam is often coupled with a boundary control, which acts on the tip and is used to dissipate the vibration. Frequently, the boundary control is realized as a suspension system, consisting typically of springs and dampers.

In the last four decades, considerable attention has been paid to the stability analysis of such systems in the literature, see e.g.~\cite{Littman:Markus,Guo,Guo:Wang,Arnold:Miletic}. Most results deal with the situation in which the control is linear, thus obtaining linear boundary conditions. In general, the respective stability analysis uses results from linear functional analysis. For a general theory on the large-time behavior of several mechanical systems with nonlinear damping in the evolution equation, we refer to \cite{chu_las_08, chu_las_10}. Classical references on attractors in (infinite-dimensional) dynamical systems are \cite{hale, temam}. However, the generalization to nonlinear boundary conditions is not straightforward in most cases, since they often require a model-dependent analysis. For the nonlinear boundary dissipation in wave and plate equations we refer to \cite{chu_las_02, chu_las_04, chu_las_07}, e.g. To the author's knowledge the only beam models considered in the literature with nonlinearities at the boundary do not comprise a rigid body attached to the tip (see for example \cite{Chentouf:Couchouron,  Conrad:Pierre,conrad_pierre_2, Coron:Novel}). A main goal of this article is to increase the analytical understanding of nonlinear beam models.

In this article we investigate an Euler-Bernoulli beam which is clamped at one end (see Figure \ref{system_visual}). The free end holds a tip mass, whose mass $m$ and moment of inertia $J$ are both positive. The controller acting on the tip consists of a spring and a damper, both nonlinear. On the one hand this model is rather simple, but it still exhibits an interesting mechanical behavior (asymptotic stability vs.~the existence of periodic orbits, depending on the value of $J>0$). On the other hand, its mathematical analysis requires a new strategy that deviates from existing techniques for nonlinear models. And the authors are confident that the analysis presented in this article can easily be extended to more complex models of this kind (see \cite{msak} for a follow-up work).

\begin{figure}[h]
	\includegraphics[trim = 40mm 205mm 30mm 25mm, clip, scale=0.82]{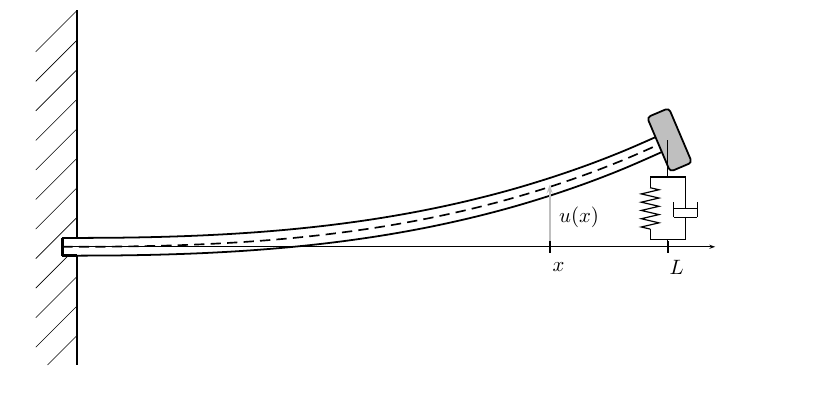}
	\caption{Clamped beam with tip mass, coupled to a spring and damper (both nonlinear)}
	\label{system_visual}
\end{figure}

%In \cite{Conrad:Pierre} the authors consider damping of an Euler-Bernoulli beam which is clamped at one end, and on the other end nonlinear feedback control point force and point bending moment are applied. Under the assumption that the nonlinear feedback is monotone dissipative, well-posedness and asymptotic stability are demonstrated. In \cite{Chentouf:Couchouron} it was shown for the same system that, under stronger conditions on feedback functions, exponential stability is ensured. 

Our beam model satisfies a linear PDE with high order nonlinear boundary conditions. In order to make the system accessible for analysis it is a common strategy to rewrite it as a nonlinear evolution equation in an appropriate (infinite-dimensional) Hilbert space $\mh$, and use the total energy of the system as a Lyapunov functional, see \cite{chu_lie_1,grabo,Kugi:Thull, Morgul2001,Villegas2009}.
%Sometimes, it is possible to directly prove asymptotic (and even exponential) stability of the system by using uniform estimates on the decay of the appropriate Lyapunov functional, cf.~\cite{morgul} for an application. 
In general, proving that every mild solution tends to zero as time goes to infinity consists of two steps, namely showing the precompactness of the trajectories and proving that the only possible limit is the zero solution. In the linear case, verifying the precompactness is straightforward by showing that the resolvent of the system operator is compact, see Remark~\ref{comp_lin} in Section~\ref{sec:4}. For the nonlinear case, the inspection of the precompactness property is more complex. In some situations, it can be shown that the (nonlinear) system operator is dissipative (in the sense of \cite{cra_pa}) and has compact resolvent, see for example \cite{conrad_pierre_2}. However, in our case, this operator is \emph{not} dissipative, and does not generate a semigoup of nonlinear contractions, see Remark~\ref{rem22} in Section~\ref{sec:3}.

A very common point of view for proving asymptotic stability of nonlinear evolution equations is to consider the system as a quasi-linear evolution equation. For this situation, the most commonly used criteria for the precompactness of trajectories can be found in \cite{daf_sle,Pazy2,Pazy3,Webb}, and further generalizations in \cite{Couchouron,Thieme_Vrabie}. They all split the system operator into the sum of two operators $A+\nn$ ($A$ being its linear, and $\nn$ its nonlinear part) and infer precompactness under the following conditions. In \cite{daf_sle} $A$ is required to be $m$-dissipative and $\nn$ applied to a trajectory is $L^1$ in time. In \cite{Pazy3} the requirement on $\nn$ is loosened by just assuming uniform local integrability of $\nn$ applied to a trajectory. However, the linear semigroup $(\e^{tA})_{t\ge 0}$ needs additionally to be compact in order to still ensure precompactness. Finally, in \cite{Webb} $\nn$ just needs to map into a compact set, but $A$ needs to generate an exponentially stable linear $C_0$-semigroup. These strategies have been successfully applied in the literature to the Euler-Bernoulli beam without tip payload and with nonlinear boundary control: in \cite{Conrad:Pierre} the precompactness of the trajectories follows directly from the $m$-dissipativity of the system operator, and in \cite{Chentouf:Couchouron} from the $L^1$-integrability of the nonlinearity.

In contrast to the mentioned  literature, the nonlinear boundary conditions considered in this article do not fall into any of these sets of assumptions. In our case $A$ shall be $m$-dissipative, but not compact and it does not generate an exponentially stable semigroup. On the other hand, $\nn$ apparently does not satisfy strong assumptions either, for it is compact, but we can not guarantee $L^1$-integrability. Thus the properties of our system operator are too weak in order to apply the mentioned standard results. However, we are still able to prove precompactness of the trajectories in a novel way. Hence, this article provides a new strategy for such evolution problems.

In this article we show that, for the Euler-Bernoulli beam with tip mass, coupled to a nonlinear spring and a nonlinear damper, all trajectories that are $C^1$ in time are precompact. Furthermore, for almost all values of the moment of inertia $J>0$ the trajectories tend to zero as time goes to infinity. Interestingly we find that, for countably many values of $J$, the trajectories tend to a time-periodic solution (see our main result, Theorem~\ref{generator}, for the detailed formulation). For given initial conditions we are able to characterize this asymptotic limit explicitly, including its phase. Such periodic limiting orbit appears when the (linear) beam equation has an eigenfunction with a node at the free end (i.e.~some $u_n(L)=0$). Then, the controller at the tip is inactive for all time.

A possible application of the method developed here is the nonlinear extension of the linear theory in \cite{BHG}, describing a model for a flexible micro-gripper used for DNA manipulation (the DNA-bundle model consists of a damper, spring and a load). Studying the stability of the system, when nonlinear phenomena for the controller and DNA-bundle are included, is a goal for future research  set  by the authors.

The paper is organized as follows. In Section \ref{sec:1} the equations of motion are derived for the system consisting of the Euler-Bernoulli beam with tip mass, connected to a nonlinear spring and damper. Next, it is shown that the energy functional is an appropriate Lyapunov function for the system. Section \ref{sec:3} is concerned with the formulation of the problem in an appropriate functional analytical setting and the investigation of existence and uniqueness of the corresponding solutions. There we also present our main result. In Section \ref{sec:4} we prove precompactness of the trajectories for all initial conditions lying in a dense subset of the underlying Hilbert space. Section \ref{sec:5} deals with the characterization of possible $\omega$-limit sets, proving that any classical solution tends either to zero or to a periodic solution, depending on the prescribed value $J$.

\section{Preliminaries and derivation of the model} \label{sec:1}

For a function $u(t,x)$, $t\ge 0$, $x\in [0,L]$ for some $L>0$, we use the notation $u_t$ for the derivative with respect to the time variable $t$, and we write $u'$ for the $x$-derivative. Higher order $x$-derivatives are also denoted by roman superscripts. Whenever it is clear from the context, we omit the time variable in the notation and write for example $u(L)\equiv u(t,L)$ and $u(0)\equiv u(t,0)$. 
If not stated otherwise all functions occurring in this article are considered to be real valued, and only real valued solutions are sought. Therefore, in addition to the Hilbert spaces $L^2(0,L)$ and $H^k(0,L)$, which are understood to consist of complex valued functions, we also need \[L^2_\R(0,L):=\{f\in L^2(0,L): f\text{ is real valued}\},\] and analogously we define $H^k_\R(0,L)$.

A linear operator $A$ is a closed, linear map $A:\XX\to \XX$, where $\XX$ is a suitable real or complex Hilbert space. The operator $A$ is defined on the domain $D(A)$ which needs to be dense in $\XX$. The range of $A$ is $\ran A\subset \XX$. A closed linear subset $X$ of a Hilbert space $\XX$ is called $A$-invariant if $X\cap D(A)\subset X$ is dense and $\ran A|_X\subset X$. 
%Depending on the context we sometimes consider operators on real Hilbert spaces. However, in order to investigate spectral properties of linear operators we need to consider them on {\em complex} spaces instead.

Throughout this article $C$ denotes a positive constant, not necessarily always the same.

\medskip

For the derivation of the mechanical model we follow \cite{Ge:Zhang:He} and \cite{Kugi:Thull}, whereby we assume that the beam satisfies the Euler-Bernoulli assumption. We assume that the beam has uniform mass per length $\rho > 0$ and length $L$. The beam is parametrized with $x \in [0, L]$, and is described by its deviation $u(t,x)$ from the horizontal (as depicted in Figure \ref{system_visual}). The constant bending stiffness is $\Lambda > 0$, and the tension is assumed to be zero. At the tip of the beam there is a payload of mass $m>0$, which has the moment of inertia $J>0$. We neglect friction of any kind. Only two external forces are assumed to act on the beam, both on the tip, perpendicular to the resting position $u\equiv 0$. The first comes from a nonlinear spring attached to the tip, producing the restoring force $-k_1(u(t,L))$. The second force is due to a nonlinear damping, and is given by $-k_2(u_t(t,L))$. Throughout the rest of the paper we shall make the following assumptions on the two nonlinearities:

\begin{assumption}\label{ass}
We assume $k_1,k_2 \in W^{2,\infty}_{\mathrm{loc}}(\R)$, and
\begin{align}
 \int_0^{z}{k_1(s) \dd s} \ge 0, & \:\:\: \forall z \in \mathbb{R},\label{assum_1} \\
  k_2'(z) \ge 0,\,\, k_2(0)=0, & \:\:\: \forall z \in \mathbb{R}. \label{assum_2}
\end{align}
Furthermore, we assume that
\begin{equation} \label{k_quadr}
|k_2(z)| \ge K z^2,\quad \forall z \in (-\delta,\delta),
\end{equation} for some positive constant $K>0$ and $\delta>0$ small. 
\end{assumption}
Notice that \eqref{assum_1} implies $k_1(0) = 0$, and \eqref{assum_2} together with \eqref{k_quadr} imply that $k_2(z)= 0$ iff $z= 0$.

\begin{remark}
The assumption $k_1,k_2\in W^{2,\infty}_{\mathrm{loc}}(\R)$ is needed in the proof of Lemma~\ref{reg_precompactness}, which specifically requires the $L^\infty$-boundedness of $k_1''$ and $k_2''$ on bounded sets.

In many situations the damping $k_2$ will originate from the drag produced by the flow of a fluid around the immersed tip. Hence, the dependence of the drag $k_2$ on the velocity of the tip will, in general, either be linear (\emph{Stokes drag}; for low Reynolds numbers), or quadratic (\emph{drag equation}; for high Reynolds numbers with turbulence behind the object), see \cite{fluid}. From this point of view, condition \eqref{k_quadr} is not restrictive.
\end{remark}

The equations of motion can be derived according to Hamilton's principle. Hence, they are the Euler-Lagrange equations corresponding to the action functional. In our model the kinetic energy $E_k$ and the potential (strain) energy $E_p$ are
\[E_k = \frac{\rho}{2} \int_0^L{u_t^2 \dd x} + \frac{m}{2}u_t(L)^2 + \frac{J}{2}u'_t(L)^2,\qquad E_p = \frac{\Lambda}{2} \int_0^L{(u'')^2 \dd x}.\]
Additionally, we have the virtual work $\delta W$ coming from the external forces:
\[\delta W =  - k_1(u(L))\delta u(L) - k_2(u_t(L)) \delta u(L).\]
Taking into account the boundary conditions $u(0)=u'(0)=0$ of the clamped end we find that, according to Hamilton's principle, $u$ satisfies:
\begin{subequations} \label{EBB_system}
\begin{align}
\rho u_{tt}(t,x) + \Lambda u^{\rom{4}}(t,x) & = 0,\\
u(t,0) = u'(t,0) & =  0, \\
-\Lambda u'''(t,L)+m u_{tt}(t,L)  & = - k_1(u(t,L)) - k_2(u_t(t,L)) , \\
\Lambda u''(t,L) + J u'_{tt}(t,L) & =  0, 
\end{align} 
\end{subequations}
where $(t,x)\in (0,\infty)\times (0,L)$.

\medskip

Finally we derive a candidate for a Lyapunov function: The total energy of the system is a natural candidate, since it will decrease in time because of the damping. The total energy is given by $E_\mathrm{tot} = E_{k} + E_{p} + E_{s}$, where $E_s := \int_0^{ u(L)}{k_1(s) \dd{}s}$ represents the potential energy stored in the nonlinear spring. Now \eqref{assum_1} ensures that this integral always stays non-negative.  We compute the time derivative of the total energy, using the Euler-Lagrange equations \eqref{EBB_system}:
\begin{align*}
\frac{\dd}{\dd t}E_\mathrm{tot} 
& =  \Lambda \int_0^Lu'' {u}''_t  \dd{}x  + \rho\int_0^L{u_{tt} {u}_{t} \dd{}x} + m  u_{tt}(L){u}_t(L) + J u'_{tt}(L) {u}'_{t}(L)  \\& \,\,\,\,+ k_1(u(L))u_t(L) \\
& =   \Lambda \int_0^L{u^{\rom{4}} {u}_{t}  \dd{}x}  + \Lambda u'' {u}'_{t}  \big|_0^L - \Lambda u''' {u}_{t} \big|_0^L + \rho\int_0^L{u_{tt} {u}_{t} \dd{}x} \\
&\quad+ m  u_{tt}(L){u}_t(L) + J u'_{tt}(L){u}'_t(L) + k_1(u(L))u_t(L)\\ 
& =  \Lambda u''(L) {u}'_{t}(L) - \Lambda u'''(L) {u}_{t}(L) + m u_{tt}(L) u_{t}(L) + J u'_{tt}(L) {u}'_{t}(L)\\& \,\,\,\,+ k_1(u(L))u_t(L) \\
& = -  k_2(u_t(L)) {u}_{t}(L).
\end{align*}
Due to \eqref{assum_2} this last line is always non-positive. So $E_{\mathrm{tot}}$ is a candidate for a Lyapunov function:
\begin{equation}\label{lyapunov}
 V(u) := \frac{\Lambda}{2} \int_0^L(u'')^2  \dd x  + 
\frac{\rho}{2} \int_0^L{u_{t}^2 \dd{}x} + \frac{m}{2} u_t(L)^2 + \frac{J}{2}u'_t(L)^2+\int_0^{ u(L)}{k_1(s) \dd{}s},
\end{equation}
and it is non-negative. According to the previous calculation its derivative along classical solutions of \eqref{EBB_system} satisfies:
\begin{equation} \label{Lyapunov_decay}
\frac{\dd}{\dd t} V(u(t)) = - k_2(u_t(t,L)) u_t(t,L)  \le 0.
\end{equation}

\section{Formulation as an evolution equation and main results} \label{sec:3}

The aim of this section is to show that, for sufficiently regular initial conditions $u(0,x)=u_0(x)$ and $u_t(0,x)=v_0(x)$, the system \eqref{EBB_system} has a unique (mild) solution $u(t,x)$.
First, we introduce the standard setting for the Euler-Bernoulli beam with a tip payload (see \cite{Kugi:Thull}, \cite{Arnold:Miletic}). To this end we define the following real Hilbert space:
\begin{equation}\label{def_mh}
\mh  
:= \{ y = [u, v, \xi, \psi]^\top \colon u \in \tilde H^{2}_{0,\R}(0, L), v \in L^{2}_\R
 (0,L),  \xi, \psi \in \R \},
\end{equation}
where $\tilde H^{n}_{0,\R}(0, L):=\{f \in H^{n}_\R(0, L): f(0)=f'(0)=0\}$ for $n\ge 2$. The space $\HH$ is equipped with the inner product
\[
 \langle y_1,  y_2  \rangle_\HH  :=  \frac\Lambda2 \int_{0}^{L} 
  u_1'' {u_2}'' \dd{}x + \frac{ \rho}{2} \int_{0}^{L} v_1 {v_2} \dd{}x + \frac{1}{2 J} \xi_1 {\xi_2} + \frac{1}{2 m} \psi_1
{\psi_2 },\quad \forall y_1,y_2\in\HH. \]
We next consider the following linear operator on $\HH$:
\begin{equation}\label{operator_A}
A:y\mapsto \begin{bmatrix}
 v \\ - \frac{\Lambda}{\rho} u^{\rom{4}} \\ - \Lambda u''(L) \\ \Lambda u'''(L) 
\end{bmatrix}
\end{equation}
on the dense domain
\begin{equation}\label{def_domain_A}
D(A) :=  \{y \in \mh \colon 
 u \in \tilde H^{4}_{0,\R}(0, L), v \in \tilde H^{2}_{0,\R} (0,L), \xi = J v'(L), \psi = m v(L) \}.
\end{equation}
Furthermore, we define the bounded nonlinear operator $\nn$ on $\HH$:
\begin{equation}\label{def_N}
\nn:y\mapsto\begin{bmatrix}
0\\0 \\ 0 \\ - k_1(u(L)) - k_2(\frac{\psi}{m})
\end{bmatrix}.
\end{equation}
\sloppy{Finally, we introduce the nonlinear operator $\A:=A+\nn$ on the domain ${D(\A)=D(A)}$. With this notation the system \eqref{EBB_system} can be written formally as the following nonlinear evolution equation in $\HH$:}
\begin{subequations} \label{ivp}
\begin{align}
y_t&=\ma y,\\
y(0)&=y_0,
\end{align}
\end{subequations}
for some initial condition $y_0 \in \mh$. A function $y(t)$  is a classical solution of \eqref{ivp} on $(0,T)$ if $y\in C^1((0,T);\HH)\cap C([0,T);\HH) $ and for all $t\in(0,T)$ there holds $y(t)\in D(\A)$ and \eqref{ivp}. A continuous function $y\in C([0,T);\HH)$ that satisfies the Duhamel formula
\begin{equation}\label{duhamel}
y(t)=\e^{tA}y_0+\int_0^t \e^{(t-s)A}\nn y(s)\dd s,\quad t\in(0,T),
\end{equation}
is called a mild solution of \eqref{ivp}, see \cite{Pazy}.

\begin{lemma}\label{lumphil}
Let $A$ be the operator defined in \eqref{operator_A}, and $\HH$ the Hilbert space \eqref{def_mh}. Then, $A$ generates a $C_0$-semigroup $(e^{tA})_{t \ge 0}$ of unitary operators in $\HH$.
\end{lemma}

The proof of Lemma~\ref{lumphil} is included in the proof of Lemma~\ref{semi_A} in the Appendix. The latter is an extension of Lemma~\ref{lumphil} to the complex analogue of $\HH$.

For the subsequent analysis we need to properly define a Lyapunov function on $\HH$ (as defined in \eqref{def_mh}). In the previous section we obtained a candidate by \eqref{lyapunov}, which was defined along classical solutions $u(t)$ of \eqref{EBB_system}. For $y=[u, v,\xi,\psi]^\top\in D(\A)$ (see \eqref{def_domain_A}) this can equivalently be written as 
\begin{equation}\label{def_V}
 V(y)= \frac{\Lambda}{2} \int_0^L(u'')^2  \dd x  + 
\frac{\rho}{2} \int_0^L{v^2 \dd{}x} + \frac{1}{2m} \psi^2 + \frac{1}{2J} \xi^2+\int_0^{ u(L)}{k_1(s) \dd{}s}. 
\end{equation}
This is then defined for all $y\in\HH$, and according to \eqref{Lyapunov_decay} there holds
\begin{equation}\label{lyap_dec_gnral}
\frac{\dd}{\dd t} V(y(t)) = - k_2\Big(\frac{\psi(t)} m\Big)\frac{\psi(t)} m  \le 0,
\end{equation}
along all classical solutions $y(t)$. The following two results are easily verified:

\begin{lemma}\label{V:cont}
Under the Assumptions~\ref{ass} the function $V \colon \mh \to \R$, defined in \eqref{def_V}, is continuous with respect to the norm $\|\cdot\|_\HH$. 
\end{lemma}
\begin{lemma}\label{v_bdd}
Let the Assumptions~\ref{ass} hold, and let $V$ be the functional from \eqref{def_V}, and $\HH$ the Hilbert space \eqref{def_mh}. Then, for any $Y \subset \mh$ we have:
\[\sup{\{ V( y ) \colon y \in Y\}} < \infty \quad \Leftrightarrow \quad \sup{\{\| y \|_{\mh} \colon y \in Y\}} < \infty\] 
\end{lemma}

\begin{proposition}\label{prop_exist}
If the Assumptions~\ref{ass} are satisfied, then for every $y_0\in\HH$ there exists a unique mild solution $y:[0,T_{\mathrm{max}}(y_0))\to\HH$ of \eqref{ivp}, where $T_{\mathrm{max}}(y_0)$ is the maximal time interval for which the solution exists. If $T_{\mathrm{max}}(y_0)<\infty$, then a blow-up occurs, i.e.~
\[\lim_{t \nearrow T_{\mathrm{max}}}{\|y(t)\|_{\mh}} = \infty.\]
\end{proposition}

\begin{proof}
Due to the Assumptions~\ref{ass} made on $k_1,k_2$ it follows that $\mathcal{N}$, given in \eqref{def_N}, is continuously differentiable on $\HH$, and thus locally Lipschitz continuous. Furthermore, $A$ generates a $C_0$-semigroup, see Lemma~\ref{lumphil}. Hence, according to Theorem 6.1.4 in \cite{Pazy} a unique mild solution exists on $[0, T_{\mathrm{max}}(y_0))$, for some maximal $0 < T_{\mathrm{max}}(y_0) \le  \infty$.
Moreover, if $T_{\mathrm{max}} <  \infty$ then $\lim_{t \nearrow T_{\mathrm{max}}}{\|y(t)\|_{\mh}} = \infty$.
\end{proof}

\begin{lemma}\label{lem_class}
If the Assumptions~\ref{ass} hold, and $y_0\in D(\A)$ (see \eqref{def_domain_A}), then the corresponding mild solution $y(t)$ of \eqref{ivp} is a classical solution. Furthermore $y(t)$ is a global solution, i.e.~$T_{\mathrm{max}}(y_0)=\infty$.
\end{lemma}

\begin{proof}
Since $\mathcal{N}$ is continuously differentiable (because of the Assumptions~\ref{ass}), Theorem 6.1.5 in \cite{Pazy} implies that $y(t)$ is a classical solution.
Therefore \eqref{lyap_dec_gnral} holds and implies: 
\[V(y(t)) \le V(y_0), \qquad\forall t\in [0, T_{\mathrm{max}}),\]
with $V$ from \eqref{def_V}. Thus, according to Lemma \ref{v_bdd}, the norm $\|y(t)\|_{\mh}$ stays uniformly bounded on $[0, T_{\mathrm{max}})$. Consequently, no blowup occurs and $T_{\mathrm{max}} =  \infty$.
\end{proof}

\noindent{}The following result is a consequence of Proposition 4.3.7 of \cite{Cazenave:Haraux}:
\begin{proposition} \label{approx_solutions}
Let the Assumptions~\ref{ass} hold, and let $y :[0,T) \rightarrow \mh$ be a mild solution of \eqref{ivp} for some $y_0 \in \mh$, and $0<T \le \infty$. Also, let $\{y_{n,0}\}_{n \in \mathbb{N}} \subset D(\mathcal{A})$ be such that $y_{n,0} \rightarrow y_0$ in $\HH$. For every $n\in\N$, denote by $y_n(t)$ the (global) classical solution of \eqref{ivp} to the initial value $y_{n,0}$. Then $y_n \rightarrow y$ in $C([0,T); \mh)$ as $n\to\infty$.
\end{proposition}

\begin{theorem} \label{global_existence}
Let the Assumptions~\ref{ass} be satisfied. Then, for every $y_0 \in \mh$ the initial value problem \eqref{ivp} has a unique global mild solution $y(t)$, which is classical if $y_0\in D(\A)$. Moreover, the function $t \mapsto V(y(t))$, defined in \eqref{def_V}, is non-increasing, and $\|y(t)\|_{\mh}$ is uniformly bounded on $\R^+_0$.
\end{theorem}

\begin{proof}
Due to the previous results it remains to show the result for $y_0\notin D(\A)$. For an approximating sequence of classical solutions $\{y_n\}_{n \in \mathbb{N}}$ as in Proposition \ref{approx_solutions}, it follows that 
\[V(y(t)) = \lim_{n \rightarrow \infty}{V(y_n(t))}, \qquad \forall t \in [0, T_{\mathrm{max}}(y_0)), \] since $V$ is continuous.
Due to \eqref{Lyapunov_decay}, we know for the classical solution $y_n(t)$ that $t \mapsto V(y_n(t))$ is non-increasing for each fixed $n \in \mathbb{N}$, i.e.
\[V(y_n(t_1)) \ge V(y_n(t_2)), \qquad 0 \le t_1 \le t_2. \]
Letting $n \to\infty$ in this inequality shows that $t \mapsto V(y(t))$ is as well non-increasing on $[0, T_{\mathrm{max}})$.
%On the other hand, if we assume that $y(t)$ has a blow-up at $t=T_{\mathrm{max}}(y_0)$, then there exists $t^{\ast} < T_{\mathrm{max}}(y_0)$ such that $V(y(t^{\ast})) > M$, for some $M > V(y_0)$. 
%This is a contradiction, therefore $T_{\mathrm{max}}(y_0) = + \infty$.
Hence no blow-up can occur, and thus $T_{\mathrm{max}}(y_0) = \infty$.
\end{proof}

\begin{remark}\label{rem22}
If we assume that $k_1$ is linear, i.e.~$k_1(u(L))=K_1\cdot u(L)$ for some ${K_1>0}$, and Assumptions~\ref{ass} hold, then the (still) nonlinear operator $\A$ is dissipative in $\HH$, with respect to the modified inner product
\[
 \langle y_1,  y_2  \rangle_{\HH,2}  :=  \frac\Lambda2 \int_{0}^{L}\!\! 
  u_1'' {u_2}'' \dd{}x + \frac{ \rho}{2} \int_{0}^{L}\!\! v_1 {v_2} \dd{}x + \frac{\xi_1 {\xi_2}}{2 J}  + \frac{\psi_1
{\psi_2 }}{2 m} +\frac{K_1\Lambda}2 u_1(L)u_2(L). \]
Then, $\A$ even generates a semigroup of nonlinear contractions (cf.~\cite{cra_pa}). In this case, the asymptotic stability of the semigroup is shown more easily, see Remark~\ref{comp_lin}. However, if $k_1$ is nonlinear we cannot find a formulation of \eqref{EBB_system} such that the system operator becomes dissipative.
\end{remark}

\begin{definition}
We define the following generalized time derivative of $V$, as defined in \eqref{def_V}, for the mild solution $y(t)$ of \eqref{ivp} to the initial value $y_0 \in \mh$:
\[\dot{V}(y_0) := \limsup_{t \searrow 0}{\frac{V(y(t)) - V(y_0)}{t}}, \] 
which may take the value $- \infty$. 
\end{definition}

\begin{corollary}\label{cor:4.9}
Under the Assumptions~\ref{ass}, the function $V \colon \mh \rightarrow \mathbb{R}$, defined in \eqref{def_V}, is a Lyapunov function for the initial value problem \eqref{ivp}.
\end{corollary}

\begin{proof}
Due to Lemma \ref{V:cont} $V$ is continuous, and according to Theorem \ref{global_existence}, we know that $t \mapsto V(y(t))$ is non-increasing for all $y_0 \in \mh$. This proves the statement.
\end{proof}

For every $y_0 \in \mh$ we define $S(t)y_0 := y(t)$, for all $t \ge 0$, where $y(t)$ is the mild solution of \eqref{ivp} corresponding to the initial condition $y_0$. Theorem 9.3.2 in \cite{Cazenave:Haraux} implies that the family $S \equiv \{ S(t) \}_{t \ge 0}$ is a strongly continuous semigroup of nonlinear continuous operators in $\mh$.

In the remaining part of the paper we investigate the asymptotic stability of the nonlinear semigroup $S$. As it turns out, the semigroup is asymptotically stable ``in most cases'', i.e.~for all but countably many values of the parameter $J$. For these exceptional values of $J$, there exist non-trivial solutions which are periodic in time and do not decay, see Lemma~\ref{ef_boundary} in Section~\ref{sec:5}. The set $\mathscr J$ of exceptional values of $J$ is explicitly given by 
\begin{equation}\label{def_J}
\mathscr J:=\Big\{ \rho \left( \frac{L}{\ell \pi} \right)^3 \frac{(-1)^\ell + \cosh{\ell \pi} }{\sinh{\ell \pi}}:\ell\in\N\Big\}.
\end{equation}
We denote the $\ell$-th entry by $J_\ell$. Now we can formulate the main result of this paper:

\begin{theorem}\label{generator}
Assume that the Assumptions~\ref{ass} hold for the nonlinearities $k_1,k_2$. Let $y_0\in D(\A)$, let $y(t)$ denote the corresponding classical solution of \eqref{ivp}, and let $\mathscr J$ be the set from \eqref{def_J}. Then there holds:
\begin{enumerate}
      \renewcommand{\theenumi}{\roman{enumi}}
\renewcommand{\labelenumi}{(\theenumi)}
\item If $J\notin\mathscr J$, then the system \eqref{ivp} is asymptotically stable with respect to $\|\cdot\|_{\HH}$, i.e.~
\[\lim_{t \rightarrow \infty}{y(t)} = 0.\]
\item If $J\in\mathscr J$, then  $y(t)$ approaches (with respect to $\|\cdot \|_\HH$) the time-periodic solution corresponding to the initial condition $\Pi^*y_0$ as $t\to\infty$. Here, $\Pi^*$ is the orthogonal projection from $\HH$ onto $\Omega$, and it is given by
\begin{equation}\label{pi_ex}
\Pi^* y=\begin{bmatrix}
\Lambda \la u'' ,u_{n^*}''\ra_{L^2}\,u_{n^*}\\
|\mu_{n^*}|^2\big(\rho\la v,u_{n^*}\ra_{L^2}+\xi u_{n^*}'(L)\big)u_{n^*}\\
J|\mu_{n^*}|^2\big(\rho\la v,u_{n^*}\ra_{L^2}+\xi u_{n^*}'(L)\big)u_{n^*}'(L)\\
0
\end{bmatrix},
\end{equation}
where $\la .,.\ra_{L^2}$ denotes the standard inner product on $L^2(0,L)$. $u_{n^*}$ is the (normalized) solution of equation \eqref{ef_boundary}, see also Lemma~\ref{eigen_eqn}. $\Omega$ is the set of all trajectories along which $V$ is constant, see \eqref{def_Omega}.
\end{enumerate}
\end{theorem}

The proof of Theorem~\ref{generator} can be found at the end of Section~\ref{sec:5}, and it is based on the results developed in the Sections~\ref{sec:4}--\ref{sec:5}.

\section{Precompactness of the trajectories}\label{sec:4}

In this section we investigate the precompactness of the trajectories of \eqref{ivp}. Thus, for given $y_0\in \HH$ the corresponding trajectory $\gamma(y_0)\subset\HH$ is defined by
\[\gamma(y_0):=\bigcup_{t\ge 0}S(t)y_0,\]
where $S$ is the semigroup defined after Corollary~\ref{cor:4.9}.

First, we prove the precompactness of trajectories that are twice differentiable (in time), and then extend this result to all classical solutions.

\begin{lemma} \label{reg_precompactness}
Under the Assumptions~\ref{ass}, the trajectory $\gamma(y_0)$ of \eqref{ivp} is precompact for every $y_0 \in D(\mathcal{A})$. 
\end{lemma}

\begin{proof}
We fix $y_0\in D(\ma)$ and show that the corresponding trajectory $\gamma(y_0)$ is precompact in $\HH$. As seen in Lemma \ref{lem_class}, the solution corresponding $y(t)$ is classical. Due to the compact embeddings $H^4(0,L)\inj\inj H^2(0,L)\inj\inj L^2(0,L)$, for the precompactness of $\gamma(y_0)$ it is sufficient to show that
\[\sup_{t>0}\|\ma y(t)\|_\HH< \infty.\]
Since $y_t=\ma y$, this is equivalent to show that $\|y_t(t)\|_\HH$ is uniformly bounded for $t>0$.

\underline{\emph{Step 1:}}  In the first part of this proof we assume that $y_0 \in D(\ma^2)$.
According to Lemma \ref{da2}, the time derivative $y_t(t)$ of the corresponding solution is a classical solution of the system \eqref{EBB_system} differentiated  in time once:
\begin{subequations}  \label{system_time_der}
\begin{align}
   \rho  u_{ttt}  + \Lambda u^{\rom{4}}_t & =  0, \\  
 u_t (t,0) & =  0, \label{time_derb}\\ 
  u'_t (t,0) & =  0, \label{time_derc}\\ 
m  u_{ttt} (L) - \Lambda u'''_t (L) + k_1'(u(L))  u_t(L) + k_2'(u_t(L))  u_{tt}(L) &  =  0, \label{time_derd}\\  
 J   u'_{ttt} (L) + \Lambda u''_t (L) &  =  0.\label{time_dere}
\end{align}
\end{subequations}
We now evaluate the time derivative of $ V(y_t)$:
\begin{align}\label{lyapunov_time_der}
 \frac{\dd}{\dd t}V(y_t) & =  \Lambda \int_0^L u''_{tt}u''_t\dd x+ \rho \int_0^L\!\!\! u_{ttt} u_{tt}\dd x + J u'_{ttt}(L) {u}'_{tt}(L)  \nonumber \\ &\quad+ m u_{ttt}(L){u}_{tt}(L)+ k_1(u_t(L)) u_{tt}(L)\nonumber \\
& =  {u}_{tt}(L) \big(m u_{ttt}(L)  - \Lambda u'''_t(L) + k_1(u_t(L))\big) \\
&\quad +  {u}'_{tt}(L) \big(  \Lambda u''_t(L)  + J u'_{ttt}(L)  \big) \nonumber  \\
& = {u}_{tt}(L) \big( k_1(u_t(L)) - k_1'(u(L))  u_t(L) - k_2'(u_t(L)) u_{tt}(L) \big),\nonumber 
\end{align}
where we have performed partial integration in $x$ twice and used \eqref{time_derb}-\eqref{time_dere}.
We have due to \eqref{assum_2} \[- k_2'(u_t(L)) u_{tt}(L)^2 \le 0, \qquad \forall t\ge 0, \]
so after integration of \eqref{lyapunov_time_der} in time we obtain
\begin{equation}\label{int_lyap}
V(y_t(t))  \le  V(y_t(0)) + \int_0^t{  u_{tt}(\tau, L) \left[ k_1(u_t(\tau,L)) - k_1'(u(\tau, L)) u_t(\tau, L)\right]\dd{}\tau}. 
\end{equation}
The first integral on the right hand side, which is
\begin{align}
 \int_0^t{  u_{tt}(\tau, L) k_1(u_t(\tau,L)) \dd{}\tau} & =  \int_0^t{ \frac{\dd}{\dd \tau} \int_0^{ u_t(\tau,L)}{k_1(s) \dd s} \dd{}\tau} \nonumber\\
 & =  \int_0^{ u_t(t,L)}{k_1(s) \dd{}s} - \int_0^{ u_t(0,L)}{k_1(s) \dd{}s}, \label{est_1_v}
\end{align}
is uniformly bounded since $u_t(t,L)=\frac{\psi(t)}m$ is uniformly bounded, see Theorem~\ref{global_existence}.
For the remaining term in \eqref{int_lyap} we obtain
\begin{align}
 \int_0^t&u_{tt}(\tau, L)k_1'(u(\tau, L))  u_t(\tau, L) \dd\tau  =  \int_0^t \frac{\dd}{\dd \tau} \left( \frac{ (u_t(\tau,L))^2}{2}\right) k_1'(u(\tau, L)) \dd\tau \nonumber\\
  &=  \frac{ u_t(t,L)^2}{2} k_1'(u(t, L)) - \frac{ u_t(0,L)^2}{2} k_1'(u(0, L)) -\int_0^t \frac{u_t(\tau,L)^3}{2}  k_1''(u(\tau, L)) \dd\tau.\label{estmate_1}
\end{align}
Due to the Sobolev embedding $H^2(0,L)\inj C([0,L])$ we have the estimate ${|u(t,L)|\le C \|u\|_{H^2}\le C\|y\|_{\mh}}$. Therefore $k_1''(u(t, L))$ is also (essentially) uniformly bounded for $t \in [0, \infty)$, cf.~Assumptions~\ref{ass}. Together with the previously shown uniform boundedness of $u_t(t,L)$ we find that the first two terms in \eqref{estmate_1} are uniformly bounded, and for the remaining integral we get
\[
 \Big| \int_0^t{ \frac{u_t(\tau,L)^3}{2}  k_1''(u(\tau, L)) \dd{}\tau} \Big| \le  C \int_0^t{ |u_t(\tau,L)|^3 \dd{}\tau}.
\]
Due to \eqref{k_quadr}, and considering that $u_t(t,L)$ is uniformly bounded for $t\in\R$, there exists a positive constant $C>0$ such that $|k_2(u_t(t,L))| \ge C u_t(t,L)^2$ for all $t\ge 0$. This yields
\[\int_{0}^{\infty}{ |u_t(t,L)|^3 \dd{}t} \le C\int_{0}^{\infty}{ k_2(u_t(t,L))  u_t(t,L) \dd t},\] and since $\frac{\dd}{\dd t}(V(y(t))) = - k_2(u_t(t,L)) u_t(t,L)$ is integrable on $(0,\infty)$, we obtain $u_t(.\, , L) \in L^{3}(\mathbb{R}^{+})$.

Therefore, all terms in \eqref{estmate_1} are uniformly bounded. Together with the uniform boundedness of \eqref{est_1_v}  this shows in \eqref{int_lyap} that $V(y_t(t)) \in L^{\infty}(\mathbb{R}^{+})$, and therefore $t\mapsto \|y_t(t)\|_{\mh}$ is uniformly bounded, see Lemma \ref{v_bdd}. Hence, $\gamma(y_0)$ is precompact. Moreover, notice that actually
\begin{equation} \label{unif_bound}
\sup_{t\ge0}{\|y_t(t)\|_{\mh}} \le \tilde C(\|y_0\|_{\mh}, \|y_t(0)\|_{\mh}),
\end{equation}
where the constant $\tilde C$ depends continuously on $\|y_0\|_{\mh}$ and $\|y_t(0)\|_{\mh}$.

\underline{\emph{Step 2:}} 
For the second part of the proof, we take $y_0 \in D(\ma)$. According to Lemma \ref{density} there exists a sequence $\{y_{n,0}\}_{n \in \N} \subset D(\ma^2)$ such that 
\begin{equation} \label{stronger_conv}
\lim_{n \rightarrow \infty}{y_{n,0}}=y_0 \quad \text{and}\quad \lim_{n \rightarrow \infty}\A y_{n,0}=\A y_0 .
\end{equation} 
For the  approximating solutions $y_n(t):=S(t)y_{n,0}$ we have  $(y_n)_t(0)=\A y_{n,0}$ for all $n\in\N$, and \eqref{stronger_conv} thus implies
\begin{equation} \label{derivative_conv}
\lim_{n \rightarrow \infty}{(y_{n})_t(0)} = \ma y_0 \,\, \text{ in } \,\, \mh.
\end{equation}
Hence \eqref{stronger_conv} and \eqref{derivative_conv} imply that both $\{y_{n,0}\}_{n \in \N}$ and $\{(y_{n})_t(0)\}_{n \in \N}$ are bounded in $\mh$. Together with \eqref{unif_bound} this yields that 
\[
\sup_{\substack{t\ge0 \\n\in\N}}{\|(y_n)_t(t)\|_{\mh}} <\infty,
\]
i.e.~$\{(y_n)_t\}_{n\in\N}$ is bounded in $L^{\infty}(\R^+ ; \mh)$. Now the Banach-Alaoglu Theorem, see Theorem~I.3.15 in \cite{Rudin},
shows that there exists a $w \in L^{\infty}(\R^+ ; \mh)$ and a subsequence $\{y_{n_k}\}_{k \in \N}$ such that \[ (y_{n_k})_t \stackrel{\ast}{\rightharpoonup} w \,\, \text{ in } \,\,L^{\infty}(\R^+ ; \mh).\] 
So for  $z \in \mh$ and $t \ge 0$ arbitrary we have 
\[\lim_{k \rightarrow \infty}{\int_0^t \la (y_{n_k})_t(\tau) , z \ra_{\mh} \dd{}\tau} = \int_0^t \la w(\tau), z \ra_{\mh} \dd{}\tau,\] which is equivalent to
\[\lim_{k \rightarrow \infty}{ \la y_{n_k}(t) - y_{n_k}(0) , z \ra_{\mh}} =  \Big\la \int_0^t w(\tau) \dd{}\tau, z \Big\ra_{\mh}. \] 
Since $y_n(t) $ converges to $y(t)$ strongly in $L^\infty((0,T);\mh)$ for every $T>0$, we conclude from the above that
\[\la y(t) - y(0) , z \ra_{\mh}  = \Big\la \int_0^t w(\tau) \dd{} \tau, z \Big\ra_{\mh}.\]
Now, owing to $z\in\HH$ being arbitrary, it follows that 
\begin{equation} \label{der}
y(t) - y(0) = \int_0^t w(\tau) \dd{} \tau. 
\end{equation} 
Since $y\in C^1(\R^+;\mh)$, we can take the time derivative of \eqref{der}, and obtain $y_t \equiv w$. This implies $y_t \in L^{\infty}(\R^+ ; \mh)$, i.e.~$\|y_t(.)\|_{\mh}$ is uniformly bounded, which proves the precompactness of $\gamma(y_0)$.
\end{proof}

\begin{remark}\label{comp_lin}
In the linear case, i.e.~$\nn=0$, the proof of the trajectory precompactness is much simpler: For classical solutions $y(t)$ we have $Ay(t)=A\e^{tA}y_0 = \e^{tA}Ay_0$, so $Ay(t)$ is uniformly bounded. Since $A^{-1}$ is compact, this proves the precompactness for classical solutions. Since $(\e^{tA})_{t\ge0}$ is a contraction semigroup, any mild solution can be approximated uniformly by classical solutions, and the precompactness follows also for mild solutions.

In the case when $k_1$ is linear (i.e.~$k_1(u(L))=K_1\cdot u(L)$), and $k_1, k_2$ satisfy the Assumptions~\ref{ass}, the precompactness property of the trajectories can also be verified easily: If we incorporate the $k_1$-term and the linear part of $k_2$ into $A$, this operator still generates a (nonlinear) contraction semigroup (see also Remark~\ref{rem22}) with respect to the modified inner product
\[
 \langle y_1,  y_2  \rangle_{\HH,2}  :=  \frac\Lambda2 \int_{0}^{L}\!\! 
  u_1'' {u_2}'' \dd{}x + \frac{ \rho}{2} \int_{0}^{L}\!\! v_1 {v_2} \dd{}x + \frac{\xi_1 {\xi_2}}{2 J}  + \frac{\psi_1
{\psi_2 }}{2 m} +\frac{K_1\Lambda}2 u_1(L)u_2(L). \]
Furthermore, $A$ is invertible and has a compact resolvent in $\HH$. For the remaining nonlinear term we can show $\nn(y(.)) \in L^1(\R^+ ; \mh)$, using \eqref{lyap_dec_gnral} and \eqref{V_limit}. Then the prerequisites of Theorem 4 in the article by Dafermos and Slemrod \cite{daf_sle} are fulfilled, and the precompactness of the trajectories for all mild solutions follows.
\end{remark}

\section{$\omega$-limit set and asymptotic stability} \label{sec:5}

In this section we first investigate some properties of $\omega$-limit sets, which will turn out to be essential in the study of the asymptotic stability of the system.

\begin{definition}%[$\omega$-limit set]
Given the semigroup $S$, the {\em $\omega$-limit set} for $y_0 \in \mh$ is denoted by $\omega(y_0)$, and defined by:
\[ \omega(y_0):= \{y \in \mh \colon \exists \{t_n\}_{n \in \mathbb{N}} \subset \mathbb{R}^{+}, \,\, \lim_{n \to \infty}{t_n} = \infty \, \land \, \lim_{n \rightarrow \infty}{S(t_n)y_0} = y \} \] 
\end{definition}

It is possible that $\omega(y_0) = \emptyset$. According to Proposition 9.1.7 in \cite{Cazenave:Haraux} we have:

\begin{lemma}
For $y_0 \in \mh$, the set $\omega(y_0)$ is $S$-invariant, i.e.~$S(t)\omega(y_0) \subset \omega(y_0)$ for all $t \ge 0$.
\end{lemma}

According to \eqref{cor:4.9} the function $t \mapsto V(S(t)y_0)$ is monotonically (but not necessarily strictly) decreasing for any fixed $y_0 \in \mh$. Furthermore it is bounded from below by $0$. Therefore, the following
limit exists: 
\begin{equation} \label{V_limit}
 \nu(y_0) := \lim_{t \rightarrow \infty }{V(S(t)y_0)} \ge 0.
\end{equation}

\begin{lemma}
Suppose $\omega(y_0) \ne \emptyset$. Then there holds
\[V(\omega(y_0)) = \{\nu(y_0)\},\]
i.e.~$V$ takes the same value $\nu(y_0)$ on every element of $\omega(y_0)$. In particular $\dot{V}(y) = 0$ for all $y \in \omega(y_0)$. 
\end{lemma}

\begin{proof}
Let $y \in \omega(y_0)$. Then, there exists a sequence $\{t_n\}_{n \in \mathbb{N}} \subset \mathbb{R}^{+}$, with $t_n\to\infty$, such that $\lim_{n \rightarrow \infty}{S(t_n)y_0}=y$. Since 
$V$ is continuous, \[\lim_{n \rightarrow \infty}{V(S(t_n)y_0)}=V(y).\] Due to \eqref{V_limit} we have $V(y) = \nu(y_0)$, and the result follows.
\end{proof}

Therefore we may identify possible $\omega$-limit sets by investigating trajectories along which the Lyapunov function $V$ (see \eqref{lyapunov}) is constant. To this end we define $\Omega \subset \mh$ by:
\begin{equation}\label{def_Omega}
\Omega \text{ is the largest $S$-invariant set with:}\quad\Omega\subseteq\{y\in\HH:\dot V(y)=0\},
\end{equation}
where $S$ is the nonlinear semigroup defined after Corollary~\ref{cor:4.9}.
Clearly, there holds
\[
\omega(y_0)\subset\Omega,\quad\forall y_0\in\HH.
\]
Thus our focus first lies on characterizing $\Omega$.

\begin{lemma} \label{aux_A}
For every $ y_0 \in \mh$ the following holds, for all $t>0$:
\begin{equation} \label{A1}
        \int_0^t{S(s) y_0 \dd s} \in D(\mathcal{A}),
\end{equation}
and
\begin{equation} \label{A2}
         S(t)y_0 - y_0=A\int_0^t S(s)y_0\dd s + \int_0^t \nn S(s)y_0\dd s. 
\end{equation}
\end{lemma}

\noindent{}The proof of Lemma~\ref{aux_A} is deferred to the Appendix \ref{app:B}: While $\psi=0$ is a simple consequence of the shape of $\frac{\mathrm{d}}{\mathrm{d}t} V(y(t))$ in \eqref{lyap_dec_gnral} (at least for classical solutions), $u(L)=0$ is obtained from the uniform boundedness of $\|y(t)\|_\HH$.

 This result can be understood as a generalization of Theorem 1.2.4 in \cite{Pazy} to nonlinear semigroups. It even holds in the more general situation where $A$ is linear and the infinitesimal generator of a $C_0$-semigroup, and $\nn$ is differentiable, see \cite{thesis_stuerzer}. For more details on nonlinear contraction semigroups we refer to \cite{barbu}.

\begin{proposition}\label{null_set}
Let $\Omega$ be as defined in \eqref{def_Omega}. Then, for all $y= [u , v, \xi , \psi]^{\top} \in \Omega$ there holds $\psi = 0,\,u(L) = 0$.
\end{proposition}

The proof of Proposition~\ref{null_set} is deferred to the Appendix~\ref{app:B}. This result allows to represent any trajectory $\gamma(y_0) \subset \Omega$ as a solution to a simpler linear system characterizing $\Omega$. By inserting the result of Proposition \ref{null_set} in the equation \eqref{A2} we find that any mild solution $y(t)$ of \eqref{ivp} with $y(t)\in\Omega$ for all $t\ge 0$, satisfies the following system:
\begin{subequations} \label{projected_system}
\begin{align}
u(t) - u(0) & =  \int_0^t{v(s) \dd{} s},\label{ps_1} \\
v(t) - v(0) & =  - \frac{\Lambda}{\rho} \left(\int_0^t{u(s) \dd{} s}\right)^{\rom{4}},\label{ps_2} \\
\xi(t) - \xi(0) & =  - \Lambda \left(\int_0^t{u(s) \dd{} s}\right)''\bigg|_{x=L},\label{ps_3}\\
0 & =  \left(\int_0^t{u(s) \dd{} s}\right)'''\bigg|_{x=L},\label{ps_4}
\end{align} 
\end{subequations}
together with the additional boundary condition $u(t,L)=0$. We will show that this system is overdetermined. To this end we first investigate the system \eqref{projected_system} without the condition $u(t,L)=0$, and only incorporate it later.

The system \eqref{ps_1}-\eqref{ps_3} can be interpreted as a mild formulation of a linear evolution equation in a Hilbert space $\tilde \HH$:
\begin{equation} \label{projected_eq}
w_t = \mathcal{B} w, 
\end{equation}
with $w = [u , v , \xi]^\top\in\tilde\HH$. Here, $\tilde\HH$ is the Hilbert space
\[\mth  := \{ w = [u, v, \xi]^\top \colon u \in \tilde H^{2}_{0,\R}(0, L), v \in L^{2}_\R(0,L),  \xi \in \R \},
\] 
and $\mathcal{B}$ is the following linear operator in $\tilde \HH$:
\begin{equation}\label{def_B}
\mathcal{B} \left[ \begin{array}{c}
               u \\ v \\ \xi
              \end{array} \right] = 
              \left[ \begin{array}{c}
              v \\ - \frac{\Lambda}{\rho} u^{\rom 4} \\ - \Lambda u''(L)  
              \end{array} \right],
\end{equation}
with the domain
\[D(\mathcal{B}) :=  \{w \in \tilde\mh \colon 
 u \in \tilde H^{4}_{0,\R}(0, L), v \in \tilde H^{2}_{0,\R} (0,L), \xi = J v'(L),  u'''(L) = 0\},\]
 which incorporates the condition \eqref{ps_4}. The space $\tilde\HH$ is equipped with the inner product 
 \begin{align*}
 \laa w_1, w_2 \raa & :=  \frac{\Lambda}{2} \int_{0}^{L} 
 u_1''u_2'' \dd{}x + \frac{\rho}{2} \int_{0}^{L}{ 
 {v_1} {v_2} \dd{}x} + \frac{1}{2 J} {\xi_1} {\xi_2}.
\end{align*}

Due to Proposition \ref{operator_B} the operator $\mb$ is skew-adjoint (in $\tilde \XX$, i.e.~the complexification  of $\tilde \HH$, see the Appendix \ref{app:a}) and generates a $C_0$-group of unitary operators. The eigenvalues $\{\mu_n\}_{n\in\Z\setminus\{0\}}$ are purely imaginary, and come in complex conjugated pairs, i.e.~$\mu_{-n}=\overline{\mu_{n}}$. Zero is not an eigenvalue, since $\mb$ is invertible, see \cite{Kugi:Thull}. The corresponding eigenfunctions $\{\Phi_n\}_{n\in\Z\setminus\{0\}}$ form an orthonormal basis of $\tilde {\mathcal X}$. They are given by
\begin{equation}\label{PHI}
  \Phi_n= \begin{bmatrix}
u_n \\  \mu_n u_n \\  \mu_n J u_n'(L)
\end{bmatrix},
\end{equation}
where $u_n$ is the unique real-valued solution of
\begin{subequations} \label{eigen_eqn}
\begin{align}
\rho \mu_n^2 u_{n} + \Lambda u_{n}^{\rom{4}} & =  0, \label{ef1}\\
u_{n}'''(L) & =  0 ,\label{ef2}\\
 J \mu_n^2 u_{n}'(L) + \Lambda u_n''(L) & =  0,\label{ef3}
 \end{align}
\end{subequations}
where $u_n$ is normalized such that $\|\Phi_n\|_{\tilde \XX}=1$. Note that $\mu_n^2<0$. From \eqref{PHI} it is clear that $\Phi_{-n}=\overline{\Phi_n}$, and hence $u_{-n}=u_n$. For the complete spectral analysis of $\mb$ see Proposition \ref{operator_B} in the appendix. For notational simplicity we include the index $n=0$ in the following by setting $\mu_0:=0$ and $\Phi_0:=0$ and $u_0:=0$.

\medskip

Summarizing we note that a trajectory $\gamma( y_0)\subset\Omega$ satisfies the (reduced) linear system \eqref{projected_eq} and the boundary condition $u(t,L)=0$, for all $t\ge 0$. Now it will turn out that, for almost all values of $J>0$, the equation \eqref{projected_eq} plus the boundary condition $u(t,L)=0$ only has the trivial solution. For a countable set of $J>0$, however, it admits non-trivial solutions. They correspond to eigenfunctions of \eqref{eigen_eqn} having a node at $x=L$.

\begin{lemma} \label{ef_boundary}
There exists a non-trivial solution $u_n$ (for any $n\in\Z\setminus\{0\}$) of the system \eqref{eigen_eqn} that additionally satisfies $u_n(L)=0$ iff
\begin{equation}\label{1217}
 J = \rho \left( \frac{L}{\ell \pi} \right)^3 \frac{(-1)^\ell + \cosh{\ell \pi} }{\sinh{\ell \pi}},\quad\text{for some }\ell\in\N.
\end{equation}

In this case, $u_n(=u_{-n})$ is unique up to normalization and $\mu_n^2 = - \frac{\Lambda}{\rho} \left( \frac{\ell \pi }{L}\right)^4$. We shall denote the (unique) index of this particular eigenfunction by $n=n^*(\ell)>0$.
\end{lemma}

\noindent{}The proof of Lemma~\ref{ef_boundary} is deferred to the Appendix \ref{app:B}. We recall from \eqref{def_J} the definition of $\mathscr J$, which is the set of all values $J_\ell$ on the right hand side of \eqref{1217}. Concerning the $\omega$-limit set we distinguish between two situations:

\begin{theorem}\label{w-limit}
Let $\Omega$ be the set defined in \eqref{def_Omega}, and $\mathscr J$ the set from \eqref{def_J}.
\begin{enumerate}
\renewcommand{\theenumi}{\roman{enumi}}
\renewcommand{\labelenumi}{(\theenumi)}
\item\label{i} If $J\notin\mathscr J$, then $\Omega = \{ 0 \}$. 
\item\label{ii} If $J\in\mathscr J$, then 
\[\Omega=\spn_\R\{[u_{n^*},0, 0,0]^\top, [0,u_{n^*}, J u_{n^*}'(L),0]^\top\}.\] 
Here, $u_{n^*}$ is the non-trivial solution to \eqref{eigen_eqn}, for the given $J$.
\end{enumerate}

\end{theorem}

\begin{proof}
This proof closely follows the argumentation in \cite{Conrad:Pierre}. According to Proposition \ref{operator_B} in the Appendix \ref{app:a} we can write the mild solution of the linear evolution equation \eqref{projected_eq} with the initial condition $w_0\in\tilde\HH$ as
\begin{equation} \label{series_1}
w(t) = e^{t \mb} w_0 = \sum_{n \in \mathbb{Z}}\laa w_0,{\Phi_n} \raa_{\tilde{\mathcal X}} e^{\mu_n t} \Phi_n,
\end{equation}
where $\{\mu_n\}_{n\in\Z}$ are the (imaginary) eigenvalues of $\mb$, and $\Phi_n$ are the corresponding normalized eigenfunctions\footnote{Note that if $w_0\in\mth$, i.e.~$w_0$ is real valued, then the series always maps into $\mth$ again.}, see Proposition \ref{operator_B}. Here, $\laa .,.\raa_{\tilde{\mathcal X}}$ is the inner product in $\tilde\XX$, see the Appendix \ref{app:a}. We define $c_n:= \laa w_0,{\Phi_n} \raa_{\tilde{\mathcal X}}$ for all $n\in\Z$. Due to the ortho\-normality of the eigenfunctions $\{\Phi_n\}_{n\in\Z\setminus\{0\}}$ and the fact that $\{\mu_n\}_{n\in\Z}\subset\ii\R$ we have for any $N\in\N$:
\begin{equation}\label{unif_Estii}
\Big\|\sum_{|n|\ge N}c_n e^{\mu_n t} \Phi_n\Big\|^2_{\tilde{\mathcal X}} = \sum_{|n|\ge N}|c_n|^2.
\end{equation}
Due to Parseval's identity we also have $\sum_{n\in\Z}|\laa w_0,{\Phi_n} \raa_{\tilde{\mathcal X}}|^2 = \|w_0\|_{\tilde{\mathcal X}}^2$. As a consequence the right hand side in \eqref{unif_Estii} tends to zero as $N\to \infty$. So, for every $\varepsilon>0$ there exists some $N>0$ such that 
\begin{equation}\label{unifomr}
\sup_{t\ge 0} \Big\|\sum_{|n|\ge N}c_n e^{\mu_n t} \Phi_n\Big\|_{\tilde{\mathcal X}} <\varepsilon.
\end{equation}
The first component of the series \eqref{series_1} converges in $H^2(0,L)$ and therefore also in $C([0,L])$. Thus we have
\begin{equation}\label{u_boundary}
u(t,L)=\sum_{n\in\Z}c_n\e^{\mu_n t} u_n(L),\quad \forall t\ge 0.
\end{equation} 
Using this representation formula we now investigate those $u(t)$ that satisfy ${u(t,L)=0}$ for all times. We immediately find for every $N\in\N$:
\begin{align*}
\Big|\sum_{|n|\ge N} c_n\e^{\mu_n t} u_n(L)\Big|
&\le C \Big\|\sum_{|n|\ge N} c_n\e^{\mu_n t} u_n\Big\|_{H^2(0,L)} \\
&\le C \Big\|\sum_{|n|\ge N} c_n\e^{\mu_n t}\Phi_n\Big\|_{\tilde{\mathcal X}} .
\end{align*}
According to \eqref{unifomr} this implies that, for every $\varepsilon>0$, we can find an $N\in\N$ (large enough) such that 
\begin{equation}\label{supup}
\sup_{t\ge0}\Big|\sum_{n=-N}^N c_n\e^{\mu_n t}u_n(L)\Big|<\varepsilon,
\end{equation}
provided that $u(t,L)=0$ for all $t\ge 0$.

We fix now some $k\in\Z$ and $\varepsilon>0$, and select $N\in\N$ so large that ${|k|<N}$ and \eqref{supup} is satisfied.
Then we multiply the finite sum by $\e^{-\mu_k t}$ and integrate over $[0,T]$:
\begin{align*}
\frac 1T\int_0^T \sum_{n=-N}^N c_n\e^{\mu_n t}u_n(L)\e^{-\mu_k t}\dd t 
&= \sum_{n=-N}^N c_n u_n(L)\frac 1T\int_0^T\e^{(\mu_n-\mu_k)t}\dd t.
\end{align*}
Due to \eqref{supup} this expression still has modulus less than $\varepsilon$. Now we let $T\to\infty$. Since all eigenvalues $\mu_n$ of $\mathcal{B}$ are distinct (see Proposition \ref{operator_B}), all terms in the integral vanish except for the term where $n=k$, and we obtain
\[ |c_k\, u_k(L)| <\varepsilon.\]
Since $\varepsilon$ was arbitrary, we conclude
\begin{equation}\label{cond_ck}
 c_k\, u_k(L) = 0,\quad \forall k\in\Z.
\end{equation}
Now we need to distinguish between two situations: Either $J\notin\mathscr J$ or $J\in\mathscr J$.

(\ref{i}) In the first case,  due to Lemma \ref{ef_boundary},  $u_n(L) \ne 0$ for all $n\in\Z$. Then \eqref{cond_ck} implies that $ c_k=0$ for all $k\in\Z$, and consequently $w_0=w(t)\equiv 0$ for all $t>0$. Therefore $\Omega=\{0\}$.

\eqref{ii} Now we consider $J=J_\ell\in\mathscr J$. According to Lemma \ref{ef_boundary}, we have $u_k(L)= 0$ iff $k\neq \pm n^*(\ell)$. So we  get from \eqref{cond_ck} that
\begin{subequations}\label{ck}
\begin{align}
c_k&=0,\quad\forall k\in\Z\setminus\{\pm n^*(\ell)\},\label{ck_1}\\
c_{n^*}&\in\C\quad\text{arbitrary},\label{cl_2}
\end{align}
\end{subequations}
and $c_{-n^*}=\overline{c_{n^*}}$. This, combined with $\psi=0$ in $\Omega$, proves that ${\Omega=\Real\spn_\C\{[\Phi_{-n*}^\top,0]^\top,[\Phi_{n*}^\top,0]^\top\}=\spn_\R\{[u_{n^*}, 0 ,0,0]^\top, [0, u_{n^*}, Ju'_{n^*}(L),0]^\top\}}$.
\end{proof}

\begin{remark}
An alternative approach is to consider the system \eqref{ps_1}-\eqref{ps_3} together with $u(t,L)=0$, momentarily ignoring \eqref{ps_4}. The system \eqref{projected_eq} is then defined in the space
\[\tilde \HH_1:=\{w\in\tilde \HH: u(L)=0\}\]
instead of $\tilde \HH$, and $\mb$ has a different domain:
\[D_1(\mb):=\{w\in\tilde\HH_1: u\in\tilde H_{0,\R}^4(0,L), v\in\tilde H_{0,\R}^2(0,L), \xi=Jv'(L), v(L)=0\}.\]
Analogously to the Proposition \ref{operator_B} one finds that the operator $(\mb, D_1(\mb))$ is again skew-adjoint, generates a $C_0$-semigroup of unitary operators, and its eigenfunctions form an orthogonal basis of $\tilde \HH_1$. For the first component $\tilde u_n$ of the eigenfunctions we again get the representation of the form \eqref{gen_solu}. However, here we use $\tilde u_n(L)=0$ in order to determine the constants (i.e.~the $\tilde u_n$ are in general different to the ones used in the proof above). With these $\tilde u_n$ we have again a representation of the solution like in \eqref{u_boundary}, and only there we apply the remaining condition \eqref{ps_4}.

In the case $J\in\mathscr J$ we have seen (in Theorem \ref{w-limit}) that $\Omega=\Real \spn\{[\Phi_{\pm n^*},0]^\top\}$. From the definition of the $\Phi_{\pm n^*}$ we find that they are precisely the (two) common eigenfunctions of $(\mb, D(\mb))$ and $(\mb, D_1(\mb))$. We conclude that, in order to determine the $\omega$-limit set, the two approaches using either $(\mb, D(\mb))$ or $(\mb, D_1(\mb))$ are equivalent. They only differ in the order in which the boundary conditions $u'''(L)=0$ and $u(L)=0$ are applied.
\end{remark}

\noindent{}Now we have all the prerequisites to prove our main result:

\begin{theore}
Assume that the Assumptions~\ref{ass} hold for the nonlinearities $k_1,k_2$. Let $y_0\in D(\A)$, let $y(t)$ denote the corresponding classical solution of \eqref{ivp}, and let $\mathscr J$ be the set from \eqref{def_J}. Then there holds:
\begin{enumerate}
\renewcommand{\theenumi}{\roman{enumi}}
\renewcommand{\labelenumi}{(\theenumi)}
\item\label{a1i} If $J\notin\mathscr J$, then the system \eqref{ivp} is asymptotically stable with respect to $\|\cdot\|_{\HH}$, i.e.~
\[\lim_{t \rightarrow \infty}{y(t)} = 0.\]
\item\label{a1ii} If $J\in\mathscr J$, then  $y(t)$ approaches (with respect to $\|\cdot \|_\HH$) the time-periodic solution corresponding to the initial condition $\Pi^*y_0$ as $t\to\infty$. Here, $\Pi^*$ is the orthogonal projection from $\HH$ onto $\Omega$ (with $\Omega$ from \eqref{def_Omega}), and it is given by
\begin{align*}
\Pi^* y=\begin{bmatrix}
\Lambda \la u'' ,u_{n^*}''\ra_{L^2}u_{n^*}\\
|\mu_{n^*}|^2\big(\rho\la v,u_{n^*}\ra_{L^2}+\xi u_{n^*}'(L)\big)u_{n^*}\\
J|\mu_{n^*}|^2\big(\rho\la v,u_{n^*}\ra_{L^2}+\xi u_{n^*}'(L)\big)u_{n^*}'(L)\\
0
\end{bmatrix},\tag{\ref{pi_ex}}
\end{align*}
where $\la .,.\ra_{L^2}$ denotes the standard inner product on $L^2(0,L)$.
\end{enumerate}
\end{theore}

\begin{proof}
\underline{Case \eqref{a1i}:} According to Lemma \ref{reg_precompactness}, the trajectory $\gamma(y_0)$ is precompact. According to Theorem \ref{w-limit} we further have $\Omega=\{0\}$. So we can apply the LaSalle Invariance Principle, cf.~Theorem 3.64 in \cite{lgm}, which yields that $\lim_{t\to\infty}\|y(t)\|_{\HH}=0$.

\underline{Case \eqref{a1ii}:} With $J=J_\ell\in\mathscr J$, we consider $n^*(\ell)$ as in  Lemma \ref{ef_boundary}. According to the end of the proof of Theorem \ref{w-limit} the $\omega$-limit set is the (complex) span of the two vectors ${\Psi_{\pm n^*}=[\Phi_{\pm n^*}^\top, 0]^\top}$, where  $\Phi_{-n^*}=\overline{\Phi_{n^*} }$. Since $\Phi_{\pm n^*}\in D(\mb)$ we know that ${u_{n^*}'''(L)=0}$, and so the $\Psi_{\pm n^*}$ are eigenvectors of $A$ to the eigenvalues $\pm\mu_{n^*}$. We may now define the orthogonal projection (first in $\XX$, see the Appendix \ref{app:a}):
\[\Pi^*:=\la .,\Psi_{-n^*}\ra_\XX\Psi_{-n^*}+\la .,\Psi_{n^*}\ra_\XX\Psi_{n^*}.\]
According to Proposition \ref{spec_a} the eigenvectors of $A$ form an orthogonal basis of $\XX$, so $\Pi^*$ commutes with $A$, and $\XX=\Ker \Pi^*\oplus\ran\Pi^*$ is an orthogonal, $A$-invariant decomposition of $\XX$. In the following we work with the restriction of $\Pi^*$ to $\HH$, and keep the same notation. The explicit representation of $\Pi^*$ is given by \eqref{pi_ex}.

In the next step we show that $\Pi^*$ commutes with the nonlinearity $\nn$. Since the first component $u_{n^*}$ of $\Psi_{n^*}$ satisfies $u_{n^*}(L)=0$, it is clear that $\nn\Psi_{\pm n^*}=0$ and thus $\nn\Pi^*=0$. Now let $y\in\XX$. Then \[\nn y=\begin{bmatrix}0\\0\\0\\ -k_1(u(L))-k_2(\frac\psi m)\end{bmatrix},\]
and so $\Pi^*\mathcal N y=0$.

As a consequence, the decomposition $\HH=\Ker \Pi^*\oplus\ran\Pi^*$ is invariant under the nonlinear semigroup $S$ generated by $\ma$. The trajectories of $S|_{\Ker \Pi^*}$ lying in $D(\A)$ are still precompact. We know from Theorem \ref{w-limit} that any $\omega$-limit set of $S|_{\Ker \Pi^*}\subset S$ has to be a subset of $\ran\Pi^*$. But on the other hand any trajectory and limit of $S|_{\Ker \Pi^*}$ has to lie within $\Ker\Pi^*$, which is orthogonal to ${\ran \Pi^*}$. Thus the only possible $\omega$-limit set for $S|_{\Ker \Pi^*}$ is $\{0\}=\ran\Pi^*\cap\Ker\Pi^*$. And therefore $S(t)y_0$ approaches $S(t)\Pi^*y_0$ as $t\to\infty$.
\end{proof}

\begin{remark} \label{periodic_sol}
The asymptotic limit described in Theorem \ref{generator} can be computed explicitly. If
 $J = J_{\ell}$ for some $\ell \in \mathbb{N}$, it follows from \eqref{u_boundary}, \eqref{ck} and Lemma \ref{ef_boundary} that all real non-decaying solutions $u_p$ of \eqref{EBB_system} are time periodic. They are given by
\begin{equation} \label{non_decay}
u_p (t,x) = T(t) u_{n^*}(x),
\end{equation}
with the scalar function 
\[T(t) = a \cos\bigg[{ \sqrt{\frac{\Lambda}{\rho}}\left( \frac{\ell \pi}{L} \right)^2\!\! t}\bigg] + b \sin\bigg[{\sqrt{\frac{\Lambda}{\rho}} \left( \frac{\ell \pi}{L} \right)^2\!\! t}\bigg],\quad\text{for any } a, b \in \R,\]
and $u_{n^*}$ is given by \eqref{u_n_ast}.
In particular, it follows from Theorem \ref{generator} that for a given initial condition $y_0$ the solution $u$ of \eqref{EBB_system} approaches the solution $u_p$ given in \eqref{non_decay}, with the coefficients $a$ and $b$ determined by:
\[a := \Lambda \la u_0'' ,u_{n^*}''\ra_{L^2},\] and 
\[b := -\sqrt{ \frac{\Lambda}{\rho}} \left( \frac{\ell \pi}{L}\right)^2 \big(\rho\la v_0,u_{n^*}\ra_{L^2}+\xi_0 u_{n^*}'(L)\big).\]
\end{remark}

\begin{remark}
As already mentioned in Remark \ref{comp_lin} it is the nonlinear term $k_1(u(L))$, representing the spring in the model, that prevents the nonlinear operator $\ma$ from being dissipative. As a consequence the semigroup $S$ is not contractive and therefore it is not possible to extend the precompactness of the classical trajectories to the trajectories of the mild solutions using a density argument (for this see the proof of Theorem 3.65 in \cite{lgm}). From a physical point of view one might expect that (for $J\notin\mathscr J$) also the mild solutions tend to zero, which is motivated by the observation that the total energy is dissipated whenever a trajectory does not lie in $\Omega=\{0\}$, i.e.~for almost all times the system looses energy due to friction. However, from a mathematical point of view we have no information if the trajectory is precompact for a non-classical solution. Hence, it is not clear whether the trajectory converges at all as $t\to\infty$.
\end{remark}

%%%%%%%%%%%%%%%%%%%%%%%%%
%%%%%%%%%%%%%%%%%%%%%%%%%
%%%%%%%%%%%%%%%%%%%%%%%%%
%%%%%%%%%%%%%%%%%%%%%%%%%

\appendix

\renewcommand*{\thetheorem}{\Alph{section}.\arabic{theorem}}
\section{Functional analytical results}\label{app:a}

Even though the analysis of this paper is carried out for real-valued functions $u$ and as a consequence in the real Hilbert space $\HH$ (see \eqref{def_mh}), the spectral analysis of the occurring linear operators needs to be performed in a complex Hilbert space. This section contains some of those results. For the spectral analysis of the operator $A$, defined in \eqref{operator_A}, we introduce the complex Hilbert space
\[\XX:=\{y=[u,v,\xi,\psi]^\top: u\in \tilde H_0^2(0,L), v\in L^2(0,L), \xi,\psi\in\C\},\]
equipped with the inner product
\[
 \langle y_1,  y_2  \rangle_\XX  :=  \frac\Lambda2 \int_{0}^{L} 
  u_1'' \overline{u_2}'' \dd{}x + \frac{ \rho}{2} \int_{0}^{L} v_1 \overline{v_2} \dd{}x + \frac{1}{2 J} \xi_1 \overline{\xi_2} + \frac{1}{2 m} \psi_1
\overline{\psi_2 },\quad \forall y_1,y_2\in\XX. \]
For the operator $A$ given by \eqref{operator_A} we consider the natural continuation to $\XX$, still denoted by $A$. This continuation still is of the form \eqref{operator_A}, and the domain is now
\[D_\C(A) =  \{y \in \XX :
 u \in \tilde H^{4}_{0}(0, L), v \in \tilde H^{2}_{0} (0,L), \xi = J v'(L), \psi = m v(L) \},
\]
where the occurring Sobolev spaces now consist of complex valued functions.

\begin{proposition}\label{spec_a}
The linear operator $A$, given in \eqref{operator_A}, is skew-adjoint and has compact resolvent in $\XX$. The spectrum $\sigma(A)$ consists of countably many eigenvalues $\{\lambda_n\}_{n\in\Z}$. They are all isolated and purely imaginary, and each eigenspace has finite dimension. The eigenspaces form a complete orthogonal decomposition of $\XX$.
\end{proposition}

\begin{proof}
It can easily be shown that for all $y_1,y_2\in D_\C(A)$
\[
 \langle A y_1, y_2 \rangle_\XX  = \frac{\Lambda}{2}  \int_0^L v_1'' \overline{u_2}'' -  u_1'' \overline{v_2}'' \dd x  =  - \la y_1,Ay_2\ra_\XX,
\]
i.e.~$A$ is skew-symmetric. Straightforward calculations, analogous to those in \cite{Kugi:Thull}, demonstrate that $A$ is invertible and $A^{-1}:\XX\to\XX$ is even compact. So $0\in\rho(A)$, and due to the corollary of Theorem VII.3.1 in \cite{Yosida} this proves that $A$ is skew-adjoint. Then, according to Theorem III.6.26 in \cite{kato} the spectrum $\sigma(A)$ consists of countably many eigenvalues, which are all isolated. The corresponding eigenspaces are finite-dimensional, and the eigenvectors form an orthogonal basis according to Theorem V.2.10 in \cite{kato}.
\end{proof}

The following is an extension of Lemma \ref{lumphil} from $\HH$ to $\XX$.

\begin{lemma}\label{semi_A}
The linear operator $A$ generates a $C_0$-semigroup $(e^{tA})_{t \ge 0}$ of unitary operators in $\XX$.
\end{lemma}

\begin{proof}
From Proposition \ref{spec_a} we know that $A$ is skew-adjoint in $\XX$. So we may apply Stone's Theorem, and $(\e^{tA})_{t\ge 0}$ is a $C_0$-semigroup of unitary operators in $\XX$.
\end{proof}

\bigskip

In the following, we turn to the spectral analysis of $\mb$, defined in \eqref{operator_B}. To this end we introduce the Hilbert space
\[
\tilde\XX :=\{w=[u,v,\xi]^\top: u\in\tilde H_0^2(0,L), v\in L^2(0,L), \xi\in\C\},
\]
equipped with the inner product
 \begin{align*}
 \laa w_1, w_2 \raa_{\tilde \XX} & :=  \frac{\Lambda}{2} \int_{0}^{L} 
 u_1''\overline{u_2}'' \dd{}x + \frac{\rho}{2} \int_{0}^{L}{ 
 {v_1}\overline {v_2} \dd{}x} + \frac{1}{2 J} {\xi_1} \overline{\xi_2}.
\end{align*}
The continuation of $\mb$ to $\tilde \XX$ is still denoted by $\mb$ and given by \eqref{def_B}, and has the domain
\[D_\C(\mathcal{B}) :=  \{y \in \tilde\XX \colon 
 u \in \tilde H^{4}_{0}(0, L), v \in \tilde H^{2}_{0} (0,L), \xi = J v'(L),  u'''(L) = 0\}.\]

\begin{proposition} \label{operator_B}
The operator $\mb$, see \eqref{operator_B}, is skew-adjoint and has compact resolvent in $\tilde\XX$. The spectrum $\sigma(\mb)$ consists entirely of isolated eigenvalues $\{\mu_n\}_{n\in\Z\setminus\{0\}}$ located on the imaginary axis, and they have no accumulation point. All eigenspaces are one-dimensional, and the corresponding eigenfunctions form an orthogonal basis of $\tilde\XX$. The normalized eigenfunction associated to $\mu_n$ is given by 
\[
  \Phi_n= \begin{bmatrix}
u_n \\  \mu_n u_n \\  \mu_n J u_n'(L)
\end{bmatrix},\quad n\in\Z\setminus\{0\},
\]
where the real function $u_n \in \tilde{H}^4_0(0,L)$ is the unique (up to normalization) solution of the boundary value problem \eqref{eigen_eqn}. Here, $u_n$ is scaled such that $\|\Phi_n\|_{\tilde\XX}=1$.
\end{proposition}

\begin{proof}
Analogously to the proof of Proposition \ref{spec_a} we show that $0\in\rho(\mb)$, that $\mb^{-1}$ is compact in $\tilde \XX$ and that $\mb$ is skew-adjoint. Now we can apply the Corollary of Theorem VII.3.1 in \cite{Yosida}, which proves that the skew-symmetric operator $\mb$ is even skew-adjoint. According to Theorem III.6.26 in \cite{kato} the spectrum $\sigma(\mb)$ consists of countably many eigenvalues $\{\mu_n\}_{n\in\Z\setminus\{0\}}$, which are all isolated. They come in complex conjugated pairs, i.e.~$\mu_{-n}=\overline{\mu_n}$. The corresponding eigenspaces are finite-dimensional, and the eigenvectors form an orthogonal basis according to Theorem V.2.10 in \cite{kato}. Since $\mb$ is skew-adjoint we have $\sigma(\mb)\subset\ii\R$. Finally, the fact that $\{\Phi_n\}_{n\in\Z\setminus\{0\}}$ is an orthogonal basis of $\tilde \XX$ follows immediately from the application of Theorem V.2.10 in \cite{kato}.

Let $\Phi_n = [u_n , v_n , \xi_n]^{\top} \in D_\C(\mb)$ be an eigenfunction corresponding to $\mu_n$ for $n \in \Z\setminus\{0\}$,  i.e.~$\mb \Phi_n = \mu_n \Phi_n$. Now $\Phi_n$ satisfies the eigenvalue equation iff $u_n$ solves \eqref{eigen_eqn}. The $v_n$ and $\xi_n$ can be determined from $u_n$ via $v_n=\mu_nu_n$ and $\xi_n=J\mu_n u_n'(L)$. The system \eqref{eigen_eqn} has a non-trivial solution iff $\mu_n\in\sigma(\mb)$ (note that we have already shown that $0\notin\sigma(\mb)$, i.e.~we may assume $\mu_n\neq 0$). In this case we get the general solution $u_n\in \tilde H_0^4(0,L)$ of \eqref{ef1} as
\begin{equation}\label{gen_solu}
u_n(x)=C_1[\cosh px - \cos px]+C_2[\sinh px - \sin px],
\end{equation}
where $p=\big(\frac{-\rho\mu_n^2}\Lambda\big)^{\frac 14}>0$, and $C_i\in\R$. Here, we already incorporated the zero boundary conditions at $x=0$. Using the condition $u_n'''(L)=0$ from \eqref{ef2} yields
\[C_1[\sinh pL  - \sin pL]=-C_2[\cosh pL + \cos pL].\]
Since $p\neq 0$ due to $\mu_n\neq 0$, both coefficients are always nonzero. So $C_2$ can always be determined uniquely from $C_1$ via this equation. Thus, if \eqref{eigen_eqn} has a non-trivial solution, it is unique up to multiplicity. This shows that all eigenspaces of $\mb$ are one-dimensional, spanned by the $\Phi_n$. Finally, \eqref{ef3} can be used to determine the $\mu_n$ for which there is a non-trivial solution.
\end{proof}

%%%%%%%%%%%%%%%%%%%%%%%%%%%%%%%%%%%%%%%%%%%%%%%%%%%%%%%%%%%%%
%%%%%%%%%%%%%%%%%%%%%%%%%%%%%%%%%%%%%%%%%%%%%%%%%%%%%%%%%%%%%

\section{Deferred proofs and results}\label{app:B}

In Lemma~\ref{reg_precompactness}, the precompactness of the trajectories $\gamma(y_0)$ is first proven for $C^2$-trajectories. This differentiability (in $t$) follows from sufficient regularity of $y_0$ (see Lemma~\ref{da2} below). The precompactness is then extended to $C^1$-trajectories, with the required density proven in Lemma~\ref{density} below.

\begin{lemma}\label{da2}
Let $y_0 \in D(\mathcal{A}^2)$ and let $y(t)$ be the corresponding solution of \eqref{ivp}. Then $y\in C^2([0,\infty); \mathcal{H})$ and $y_t(t) \in D(\ma)$ for all $t>0$. 
\end{lemma}

\begin{proof}
First, notice that if $y\in C^2([0,\infty);\mh)$ then $\tilde y:= y_t$ would satisfy
\begin{equation}\label{ivp_2}
\tilde y_t=A \tilde y+\begin{bmatrix}
0\\0\\0\\-k_1'(u(L))\frac{\psi} m-k_2'(\frac\psi m)\frac{\tilde \psi}m
\end{bmatrix}.
\end{equation}
However, for the moment we only know that $y\in C^1([0,\infty);\mh)$, see Lemma \ref{lem_class}.
Motivated by \eqref{ivp_2} we  define the following functions for this fixed $y(t)$:
\begin{align*}
F(t) & := -k_1'(u(t,L))\frac{\psi(t)} m,\\
G(t,z)& := -k_2'\Big(\frac\psi m\Big)\frac{\chi} m\equiv g(t)\chi,
\end{align*}
where $z=[U,V,\zeta,\chi]^\top\in\mh$. Since $y(t)$ is a classical solution, and $k_1,\,k_2\in W_{\mathrm{loc}}^{2,\infty}(\R)$ due to the Assumptions~\ref{ass}, both $F(t)$ and $g(t)$ lie in $W_{\mathrm{loc}}^{1,\infty}(\R)$. Consequently, the operator $\tilde\nn :[0,T]\times\HH\to\HH$, defined by
$\tilde\nn (t,z):=[0,0,0, F(t)+G(t,z)]^\top$, is Lipschitz continuous in both variables, for every $T>0$. In the following we consider the (non-autonomous) initial value problem
\begin{subequations}\label{ivp_z}
\begin{align}
z_t &= Az+\tilde\nn(t,z),\\
z(0)&=z_0\in\mh.
\end{align}
\end{subequations}
We apply Theorem 6.1.2 in \cite{Pazy} which proves that there is a unique global mild solution $z(t)$ of \eqref{ivp_z} for every $z_0\in \HH$. Furthermore, if $z_0\in D(A)$ then $z(t)$ is Lipschitz continuous on $[0,T]$, which follows from the proof of Theorem 6.1.6 in \cite{Pazy}. Consequently, $f\colon t\mapsto \tilde\nn(t,z(t))$ is also Lipschitz continuous on $[0,T]$. We then consider \eqref{ivp_z} where we replace $\tilde\nn$ by $f$. Now, $z(t)$ is also a mild solution of this reformulated evolution problem.  But, according to Corollary~4.2.11 in \cite{Pazy} (since $f$ is Lipschitz and $\HH$ is reflexive), $z(t)$ is even a classical solution of this modified problem. So, $z(t)$ is also a classical solution of \eqref{ivp_z}.

We next show that for the given classical solution $y(t)$ the function $y_t(t)$ is a mild solution of \eqref{ivp_z} for $z_0=\A y_0$. Clearly, $y(t)$ satisfies the Duhamel formula \eqref{duhamel}, and differentiation with respect to $t$ yields
\begin{equation}\label{diffmild}
y_t(t)= \e^{tA}A y_0+\frac{\dd}{\dd t}\int_0^t\e^{(t-s)A}\nn y(s)\dd s.
\end{equation}
According to the proof of Corollary 4.2.5 in \cite{Pazy} there holds
\[\frac{\dd}{\dd t}\int_0^t\e^{(t-s)A}\nn y(s)\dd s=\e^{tA}\nn y_0+\int_0^t\e^{(t-s)A}\frac{\dd}{\dd s}\nn y(s)\dd s.\]
Inserting this in \eqref{diffmild} proves that $y_t(t)$ fulfills the Duhamel formula for \eqref{ivp_z}, and as a consequence $y_t(t)$ is the unique mild solution of \eqref{ivp_z} to the initial condition $z_0=\A y_0$. But from the first part of the proof we know that this mild solution $z(t)=y_t(t)$ is a classical solution of $\eqref{ivp_z}$ if $\A y_0\in D(\A)$, i.e.~$y_0\in D(\A^2)$. So $y_t\in C^1(\R^+;\HH)$ and $y\in C^2(\R^+;\HH)$.
\end{proof}

\begin{remark}
In the situation where the evolution equation is linear and autonomous, i.e.~$\nn=0$ in our case, the above result is straightforward. If $y_0\in D(A^2)$, then we have according to Section II.5.a in \cite{engelnagel} that $y(t)\in D(A^2)$ for all $t\ge 0$. Therefore $y_t(t)=Ay(t)\in D(A)$, and so $y_{tt}=Ay_t = A^2 y$, and $y_{tt}\in C(\R^+)$. There it is crucial that the time derivative and the operator $A$ on the right hand side commute. This does not hold in the nonlinear situation any more, which makes the proof more complicated. According to Section II.5.a in \cite{engelnagel} the density of $D(A^2)$ in $\HH$ is also immediate. In our case $D(\A^2)$ is a nonlinear subset of $\HH$, see \eqref{night6}, so we need to check the density separately.
\end{remark}

We even show the stronger property that $\A|_{D(\A^2)}\subset \A|_{D(\A)}$ is dense in the product topology of $\HH\times \HH$.

\begin{lemma} \label{density}
For any $y \in D(\ma)$, there exists a sequence $\{y_{n}\}_{n \in \N}$ in $D(\ma^2)$ such that $\lim_{n \rightarrow \infty}{y_n} =y$ and $\lim_{n \rightarrow \infty}{\ma y_n} = \ma y$ in $\mh$.
\end{lemma}

\begin{proof}
First we characterize $D(\ma^2)$. We use that $y\in D(\ma^2)$ if and only if $y\in D(\ma)$ and  $\ma y\in D(\ma)$, or equivalently
\begin{align}
v &\in \tilde H_{0,\R}^4(0,L),\label{night4}\\
u\in \tilde H_{0,\R}^6(0,L) \,\wedge \, u^{\rom{4}}(0)&= u^{\rom{5}}(0)=0,\label{night5}\\
\xi&=Jv'(L),\label{night1}\\
\psi &=mv(L),\label{night2}\\
 u''(L)&= \frac{J}{\rho}u^{\rom{5}}(L),\label{night3}\\
\Lambda u'''(L)-k_1(u(L))-k_2\big(\frac\psi m\big)&=-\frac{m\Lambda}{\rho}u^{\rom{4}}(L).\label{night6}
\end{align}
It suffices to show that for an arbitrary $y\in D(\ma)$ we can construct ${\{y_n\}_{n \in \mathbb{N}} \subset D(\ma^2)}$ such that $y_n = [ u_n, \, v_n, \, \xi_n, \, \psi_n]^{\top}$ converges to $y$ in the space $H^4(0,L) \times H^2(0,L) \times \R^{2}$.
Since $\tilde{C}^{\infty}_0(0,L):=\{f \in C^{\infty}(0, L) \colon f^{(k)}(0) = 0,\allowbreak \forall k \in   \N \cup\{0\}\}$ is dense in $\tilde{H}^2_0(0,L)$ (see Theorem 3.17 in \cite{Adams}), there exists a sequence $\{v_n\}_{n \in \mathbb{N}} \subset \tilde{C}^{\infty}_0(0, L)$ such that $\lim_{n \to\infty}{v_n} = v$ in $H^2(0,L)$. Clearly $v_n\in\tilde H_{0,\R}^4(0,L)$ for all $n\in\N$. Defining $\xi_n := J v_n'(L)$ and $\psi_n := m v_n(L)$ ensures that $y_n$ satisfies \eqref{night1} and \eqref{night2}. Moreover, the Sobolev embedding $H^2(0,L) \inj C^1([0,L])$ implies that $\lim_{n \to\infty}{\xi_n} = \xi$ and $\lim_{n \to\infty}{\psi_n} = \psi$.

As a final step, we construct a sequence $\{ u_n \}_{n \in \mathbb{N}} \subset C^{\infty}(0,L)$ such that $u_n$ satisfies \eqref{night5}, \eqref{night3}, and \eqref{night6} for all $n \in \mathbb{N}$, and $\lim_{n \to \infty}{u_n} = u$ in $H^4(0,L)$.
For this purpose, first we introduce the polynomial, for $n\in\N$: 
\[h_n(x) := h_{2,n} x^2 + h_{3,n} x^3 + h_{6,n} x^6 + h_{7,n} x^7 + h_{8,n} x^8 + h_{9,n} x^9 + h_{10,n} x^{10} + h_{11,n} x^{11}.\] We next show that the coefficients $h_{2,n}, \dots, h_{11,n} \in \R$ can be uniquely determined for every $n\in\N$, given certain boundary conditions on $h_n$. In the following we fix $n\in\N$ arbitrary. From the definition of $h_n$ it is already immediate that 
\begin{equation} \label{bc_0}
h_n(0) = h_n'(0) = h_n^{\rom{4}}(0) = h_n^{\rom{5}}(0) = 0,
\end{equation} holds. Then we set $h_{2,n} = \frac{u''(0)}{2}$ and $h_{3,n} = \frac{u'''(0)}{6}$, which is equivalent to 
\begin{equation} \label{bc_00}
h_n''(0) = u''(0),\,\, h_n'''(0) = u'''(0).
\end{equation}
Assume further that 
\[
h_n^{(k)}(L) = u^{(k)}(L), \qquad k \in \{0,1,2,3\},
\]
which reads equivalently\footnote{Here we use the notation $k^{\underline{l}} : = k!/(k-l)!$ for $k,l \in \N$, $k\ge l$.}:
\begin{subequations} \label{cond1}
\begin{align}
  r_1&=h_{n,6} + h_{n,7} L + h_{n,8} L^2 + h_{n,9} L^3 + h_{n,10} L^{4} + h_{n,11} L^{5} \\
 r_2&=6 h_{n,6} + 7 h_{n,7} L + 8 h_{n,8} L^2 + 9 h_{n,9} L^3 + 10 h_{n,10} L^{4} + 11 h_{n,11} L^{5} \\
 r_3&=6^{\underline{2}} h_{n,6} + 7^{\underline{2}} h_{n,7} L + 8^{\underline{2}} h_{n,8} L^2 + 9^{\underline{2}} h_{n,9} L^3\nonumber\\&\qquad + 10^{\underline{2}} h_{n,10} L^{4} + 11^{\underline{2}} h_{n,11} L^{5}  \\
 r_4&=6^{\underline{3}} h_{n,6} + 7^{\underline{3}} h_{n,7} L + 8^{\underline{3}} h_{n,8} L^2 + 9^{\underline{3}} h_{n,9} L^3 \nonumber\\&\qquad+ 10^{\underline{3}} h_{n,10} L^{4} + 11^{\underline{3}} h_{n,11} L^{5} ,
\end{align}
\end{subequations}
where 
\begin{align*}
r_1 =  \frac{u(L)}{L^6} - \frac{u''(0)}{2 L^4} - \frac{u'''(0)}{6 L^3},&\quad
r_2  =  \frac{u'(L)}{L^5} - \frac{u''(0)}{L^4} - \frac{u'''(0)}{2 L^3},\\
r_3  =  \frac{u''(L)}{L^4} - \frac{u''(0)}{L^4} - \frac{u'''(0)}{L^3},&\quad
r_4  =  \frac{u'''(L)}{L^3} - \frac{u'''(0)}{L^3}.
\end{align*}
Finally the two additional conditions are imposed on $h_n$: 
\begin{align}
\frac{m \Lambda}{\rho} h_n^{\rom{4}}(L) &= -\Lambda u'''(L) + k_1(u(L)) + k_2(\frac{\psi_n}{m}),\label{bc_1}\\
\frac{J}{\rho} h_n^{\rom{5}}(L) &= u''(L).\label{bc_11}
\end{align}
\eqref{bc_1} and \eqref{bc_11} are equivalent to:
\begin{subequations} \label{cond2}
\begin{align}
6^{\underline{4}} h_{n,6} + 7^{\underline{4}} h_{n,7} L + 8^{\underline{4}} h_{n,8} L^2 + 9^{\underline{4}} h_{n,9} L^3 + 10^{\underline{4}} h_{n,10} L^{4} + 11^{\underline{4}} h_{n,11} L^{5} & = r_5, \\
6^{\underline{5}} h_{n,6} + 7^{\underline{5}} h_{n,7} L + 8^{\underline{5}} h_{n,8} L^2 + 9^{\underline{5}} h_{n,9} L^3 + 10^{\underline{5}} h_{n,10} L^{4} + 11^{\underline{5}} h_{n,11} L^{5} & = r_6,
\end{align}
\end{subequations}
with
\begin{align*}
r_5 =  \rho \frac{- \Lambda u'''(L) + k_1(u(L)) + k_2(\frac{\psi_n}{m})}{\Lambda m L^2},&\quad
r_6  =  \frac{\rho u''(L)}{J L}.
\end{align*}
The linear system consisting of \eqref{cond1} and \eqref{cond2} has a strictly positive determinant. Hence, its solution $h_n$ exists and is unique.
Consequently, \eqref{bc_0}, \eqref{bc_00}, and \eqref{cond1} imply that $u - h_n \in H^4_0(0,L)$, for all $n \in \N$. Since $C^{\infty}_0(0,L)$ is dense in $H^4_0(0,L)$, there exists a sequence $\{\tilde{u}_n\}_{n \in \mathbb{N}} \subset C^{\infty}_0(0,L)$ such that $\|\tilde{u}_n - (u - h_n)\|_{H^4} < \frac{1}{n}$, $\forall n \in \N$.  Now defining $u_n = \tilde{u}_n + h_n$, gives $\lim_{n \to \infty}{u_n} = u$ in $H^4(0,L)$. Obviously $u_n$ satisfies \eqref{night5} for all $n \in \mathbb{N}$. Also due to \eqref{bc_1} and \eqref{bc_11}, $u_n$ satisfies \eqref{night3} and \eqref{night6}, as well. The statement follows.
\end{proof}

\begin{proof}[\underline{Proof of Lemma \ref{aux_A}}]
We first consider $y_0\in D(\A)$. Then $S(t)y_0$ is the classical solution of \eqref{ivp}, and satisfies the integrated equation:
\[S(t)y_0-y_0=\int_0^t AS(s)y_0\dd s + \int_0^t \nn S(s)y_0\dd s.\]
Since $S(t)y_0\in C^1(\R^+;\HH)$ and $\nn$ is locally Lipschitz continuous, we find that both $t\mapsto \nn S(t)y_0$ and $t\mapsto \mathcal AS(t)y_0$ are continuous, so $AS(t)y_0\in C(\R^+;\HH)$. Therefore we may write for any $t>0$:
\begin{align*}
\int_0^t S(s)y_0\dd s &= \lim_{N\to\infty}\sum_{j=1}^N\frac tNS\Big(\frac {jt}{N}\Big)y_0,\\
\int_0^t AS(s)y_0\dd s &= \lim_{N\to\infty}\sum_{j=1}^N\frac tNAS\Big(\frac {jt}{N}\Big)y_0=\lim_{N\to\infty}A\sum_{j=1}^N\frac tNS\Big(\frac {jt}{N}\Big)y_0,
\end{align*}
due to the linearity of $A$. Since $A$ is skew-adjoint, see Proposition \ref{spec_a}, it is closed. So we obtain that
\[
\int_0^t S(s)y_0\dd s \in D(A),\quad A\int_0^t S(s)y_0\dd s=\int_0^t AS(s)y_0\dd s.
\] So there holds \eqref{A2} for $y_0\in D(\A)$ and any $t>0$.

Now let $y_0\in\HH\setminus D(\A)$, and $\{y_{n,0}\}\subset D(\A)$ such that $y_{n,0}\to y_0$ as $n\to\infty$. For every $T>0$ we have $S(t)y_{n,0}\to S(t)y_0\in C( [0,T],\HH)$. Since $\nn$ is locally Lipschitz continuous, we get for every $t>0$:
\begin{align*}
\lim_{n\to\infty}(S(t)y_{n,0}-y_{n,0}) &= S(t)y_0-y_0,\\
\lim_{n\to\infty}\int_0^t\nn S(s)y_{n,0}\dd s &= \int_0^t\nn S(s)y_{0}\dd s .
\end{align*}
Applying those two limits in \eqref{A2} for $y_{n,0}$ we obtain:
\begin{align*}
\lim_{n\to\infty}A\int_0^tS(s)y_{n,0}\dd s &= S(t)y_0 - y_0- \int_0^t \nn S(s)y_0\dd s.
\end{align*}
But there also holds \[\lim_{n\to\infty}
\int_0^tS(s)y_{n,0}\dd s = \int_0^t S(s)y_{0}\dd s.\]
Since $A$ is closed, these last two limits prove \eqref{A1} and \eqref{A2} for $y_0\in \HH\setminus D(\A)$.
\end{proof}

\begin{proof}[\underline{Proof of Proposition~\ref{null_set}}]
For a fixed $y_0 \in \Omega$ let $y(t) = S(t)y_0$. Since $\Omega$ is $S$-invariant we have $ V(y(t)) = \nu(y_0)$ for all $t \ge 0$.
First we show that
\begin{equation} \label{V_null}
                    \psi(t) = 0, \qquad \forall t\ge0.
                   \end{equation}
In the case when $y_0 \in \Omega \cap D(\mathcal{A})$, \eqref{V_null} follows easily since \eqref{Lyapunov_decay} implies for the corresponding classical solution \[\dot{V}(y(t)) = 0 \quad \Leftrightarrow \quad \psi(t) = 0.\]
We next investigate the case when $y_0 \in \Omega \setminus D(\mathcal{A})$. Then there is a sequence $\{y_{n,0}\}_{n \in \mathbb{N}}\subset D(\ma)$ such that $\lim_{n \rightarrow \infty}{y_{n,0}} = y_0$ in $\HH$. Theorem \ref{global_existence} implies ${y_n(t) \to y(t)}$ in $C([0,T] ; \mh)$ for any $T>0$, where $y_n(t) = S(t)y_{n,0}$. Since $V$ is locally Lip\-schitz continuous in $\mh$,
$\left\{ V(y_n(t))\right\}_{n \in \mathbb{N}}$ is a Cauchy sequence in $C([0,T];\mathbb{R})$. On the other hand we also have the convergence
\begin{equation} \label{psi_conv}
\psi_n(t) \rightarrow \psi(t) \,\,\, \text{ in } \,\,\,C([0,T] ; \mathbb{R}).
\end{equation}
Due to \eqref{Lyapunov_decay} this implies that 
\[\Big\{ \frac{\dd}{\dd t}V(y_n(t))\Big\}_{n \in \mathbb{N}}\] is a Cauchy sequence in $C([0,T]; \mathbb{R})$. We conclude that $\left\{ V(y_n(t))\right\}_{n \in \mathbb{N}}$ is a Cauchy sequence in $C^1([0,T];\mathbb{R})$.
So there exists a unique $w \in C^1([0,T];\mathbb{R})$ such that 
\begin{equation} \label{c_1_conv}
 V(y_n(t)) \rightarrow w(t) \,\,\, \text{ in } \,\,\, C^1([0,T]; \mathbb{R}).
\end{equation}
On the other hand, we know that $V(y_n(t)) \rightarrow V(y(t)) = \nu(y_0)$ for every $t \ge 0$, and hence $w(t) \equiv \nu(y_0)$. This, combined with \eqref{c_1_conv}, implies 
${\dot{V}(y_n(t)) = - k_2(\frac{\psi_n}{m}) \frac{\psi_n}{m} \to 0}$ uniformly on $[0,T]$. With \eqref{psi_conv} this now yields \eqref{V_null} and in particular $\psi(0)=0$.

We now show that $u(t,L) = 0$ for all $t \ge 0$. From \eqref{A1} and \eqref{V_null} it follows that
\[ m \left( \int_0^t{v(s) \dd{}s} \right) \bigg|_{x=L} = \int_0^t{\psi(s) \dd{}s} = 0.\]
Using this, the first component of \eqref{A2} implies
\[0 = \left( \int_0^t{v(s) \dd{}s} \right) \bigg|_{x=L} = u(t,L) - u(0,L).\]
Therefore $u(t, L)$ is constant along $y(t)$, which implies
\begin{equation}\label{lin_growth}
\int_0^t u(s,L)\dd s = u_0(L) t,\quad t\ge 0.
\end{equation}
Since $\sup_{t>0}\|y(t)\|_\HH<\infty$, it follows that $\sup_{t>0}\|v(t)\|_{L^2(0,L)}<\infty$. Therefore, the second component of \eqref{A2} implies
\begin{equation}\label{dot_h_4}
\sup_{t \ge 0}{\Big\|\Big( \int_0^t{u(s) \dd{} s}\Big)^{\rom{4}}\Big\|_{L^2(0,L)}} < \infty.
\end{equation}
We next apply the following Gagliardo-Nirenberg inequality (cf.~\cite{nirenberg}), which guarantees the existence of $C>0$ such that there holds for all $t\ge 0$:
\begin{equation}\label{gag_nir}
\Big\|\int_0^t u(s) \dd s\Big\|_{L^\infty(0,L)} \le C \Big\|\Big( \int_0^t u(s) \dd{} s \Big)^{\rom{4}}\Big\|_{L^2(0,L)}^{\frac 18}\Big\|\int_0^t u(s) \dd{} s\Big\|_{L^2(0,L)}^{\frac 78}.
\end{equation}
The first factor on the right hand side is uniformly bounded due to \eqref{dot_h_4}. For the second factor we observe that, according to Theorem \ref{global_existence}, the map $t\mapsto\|u(t)\|_{L^2(0,L)}$ is uniformly bounded, and therefore $t\mapsto \|\int_0^t u(s)\dd s\|_{L^2(0,L)}$ grows at most linearly. Altogether this implies in \eqref{gag_nir} that the function $t\mapsto\int_0^t u(s,L) \dd s$ grows at most like $t^{\frac 78}$. But this contradicts \eqref{lin_growth} unless $u_0(L)=0$. This shows that $u(t,L)=0$ for all $t\ge 0$.
\end{proof}

\begin{proof}[\underline{Proof of Lemma \ref{ef_boundary}}]
The general solution $\varphi \in \tilde{H}^4_0(0,L)$ to \eqref{ef1} is of the form \eqref{gen_solu}, with $p=\big(\frac{-\rho\mu^2}\Lambda\big)^{\frac 14}>0$. The boundary conditions \eqref{ef2} and  \eqref{ef3} are now equivalent to the following two equations for $C_1$ and $C_2$:
\begin{equation} \label{ef2_coeff}
C_1 \left( \sinh{p L} - \sin{p L} \right) + C_2 \left( \cosh{p L} + \cos{p L} \right) = 0,
\end{equation}
and
\begin{align} \label{ef3_coeff}
C_1 \left[  J \mu^2 \left( \sinh{p L} + \sin{p L} \right) + p \Lambda \left( \cosh{p L} + \cos{p L} \right) \right] &  \nonumber \\ + C_2 \left[  J \mu^2 \left( \cosh{p L} - \cos{p L} \right) + p \Lambda \left( \sinh{p L} + \sin{p L} \right) \right] &= 0.
\end{align}
Furthermore, the additional condition $\varphi(L) = 0$ reads
\begin{equation} \label{ef4_coeff}
C_1 \left( \cosh{p L} - \cos{p L} \right) + C_2 \left( \sinh{p L} - \sin{p L} \right) = 0.
\end{equation}
We first use \eqref{ef2_coeff} and \eqref{ef4_coeff} to determine the constants $C_1$ and $C_2$. In order for $\varphi$ to be non-zero the determinant of the linear system formed by \eqref{ef2_coeff} and \eqref{ef4_coeff} needs to vanish, i.e.
\begin{align*}
(\sinh{p L} - \sin{p L})^2 - (\cosh{p L} - \cos{p L})&(\cosh{p L} + \cos{p L})  \\ 
&=-2 \sinh{p L} \sin{p L} = 0.
\end{align*}
Since $pL>0$, this is true iff $p = \frac{\ell \pi}{L}$ for some $\ell \in \mathbb{N}$. Hence $\mu^2 = -{\frac{\Lambda}{\rho}} \left( \frac{\ell \pi}{L} \right)^4$. Now \eqref{ef2_coeff} gives $C_2 = - C_1 \frac{\sinh{\ell \pi} }{\cosh{\ell \pi} + (-1)^\ell}$. Now we investigate in which situation also the third condition \eqref{ef3_coeff} is fulfilled. Multiplying \eqref{ef3_coeff} by $\frac{(-1)^\ell \cosh{\ell \pi} + 1}{2 C_1}$, we get 
\begin{align*}
-J{\frac{\Lambda}{\rho}} \left( \frac{\ell \pi}{L} \right)^4\sinh \ell\pi +\frac{\ell\pi\Lambda}L[\cosh \ell\pi+(-1)^\ell]=0,
\end{align*}
and equivalently 
\[
J = \rho \left( \frac{L}{\ell \pi} \right)^3
\frac{\cosh{\ell \pi} + (-1)^\ell}{\sinh{\ell \pi}}.
\]
In this case, the eigenfunction $\varphi$ is given by (up to normalization) 
\begin{equation} \label{u_n_ast}
\varphi(x) = \Big( \cosh{\frac{\ell \pi x}{L}} - \cos{\frac{\ell \pi x}{L}}  \Big) - \frac{\sinh{\ell \pi}}{\cosh{\ell \pi} + (-1)^\ell} \Big( \sinh{\frac{\ell \pi x}{L}} - \sin{\frac{\ell \pi x}{L}} \Big).
\end{equation}
\end{proof}

\bibliographystyle{AIMS}\providecommand{\href}[2]{#2}
\providecommand{\arxiv}[1]{\href{http://arxiv.org/abs/#1}{arXiv:#1}}
\providecommand{\url}[1]{\texttt{#1}}
\providecommand{\urlprefix}{URL }

\end{document}